\documentclass[11pt,oneside]{amsart}

\usepackage{amsmath,amsthm,amssymb}
\usepackage{enumitem}

\usepackage[a4paper, total={6in, 9.5in}]{geometry}

\usepackage[utf8]{inputenc}
\usepackage[english]{babel}

\numberwithin{equation}{section}

\newtheorem{theorem}{Theorem}[section]
\newtheorem{proposition}[theorem]{Proposition}
\newtheorem{lemma}[theorem]{Lemma}
\newtheorem{corollary}[theorem]{Corollary}

\theoremstyle{definition}

\newtheorem{remark}[theorem]{Remark}

\date{}

\def\N{\mathbb N}
\def\R{\mathbb R}
\def\C{\mathbb C}
\def\Z{\mathbb Z}
\def\DLM{T^{-\lambda}_{q}L_p}
\def\DLMinfty{T^{-\lambda}_{\infty}L_p}

\begin{document}

\title{Fourier inequalities in Morrey and Campanato spaces}

\author{Alberto Debernardi Pinos}
\address{A. Debernardi Pinos \\
Universitat Aut\`onoma de Barcelona\\
Departament de Matem\`atiques\\
Campus de Bellaterra, Edifici C\\
08193 Bellaterra (Barcelona)\\
Spain} \email{adebernardipinos@gmail.com}

\author{Erlan Nursultanov}
\address{E. Nursultanov \\
 Lomonosov Moscow State University (Kazakh Branch) and
Gumilyov Eurasian National University, Munatpasova 7
\\010010
Astana Kazakhstan} \email{er-nurs@yandex.ru}

\author{Sergey Tikhonov}

\address{S. Tikhonov \\
 Centre de Recerca Matem\`atica\\
Campus de Bellaterra, Edifici~C 08193 Bellaterra (Barcelona), Spain; ICREA, Pg.
Llu\'is Companys 23, 08010 Barcelona, Spain, and Universitat Aut\`onoma de
Barcelona.}
\email{stikhonov@crm.cat}

\date{\today}



\thanks{The first author was partially supported by the Catalan AGAUR, through the Beatriu de Pin\'os program (2021BP  00072), and by CIDMA, through the Portuguese FCT (UIDP/04106/2020 and UIDB/04106/2020).
 The third author was partially supported
 by grants PID2020-114948GB-I00, 2021 SGR 00087,
 by the CERCA Programme of the Generalitat de Catalunya, and by the
Spanish State Research Agency, through the Severo Ochoa and Mar{\'\i}a de
Maeztu Program for Centers and Units of Excellence in R$\&$D
(CEX2020-001084-M).
The work  was partially supported by the Ministry of Education and Science of the Republic of Kazakhstan (grants AP14870758, AP14870361).}

\begin{abstract}
	We study  norm inequalities for the Fourier transform, namely,
	\begin{equation}\label{introduction}
		\|\widehat f\|_{X_{p,q}^\lambda} \lesssim
		\|f\|_{Y},
	\end{equation} where $X$ is either a Morrey or Campanato space and $Y$ is an appropriate function space.
	In the case of the Morrey space we sharpen  the estimate
	$
	\|\widehat f\|_{M_{p,q}^\lambda}
	\lesssim
	\|f\|_{L_{s',q}},$ $ s\geq 2,$
	$\frac{1}{s} = \frac{1}{p}-\frac{\lambda}{n}.$
	We also show that \eqref{introduction} does not hold when both $X$ and $Y$ are Morrey spaces.
	
	If $X$ is a Campanato space, we prove that \eqref{introduction} holds for $Y$ being the truncated Lebesgue  space.
\end{abstract}

\maketitle

	\section{Introduction}
	Let $0<p<\infty$ and $0\leq \lambda\leq \frac{n}{p}$. We say that a function $f$
 belongs to the Morrey space $M_p^\lambda=M_p^\lambda (\R^n)$ if $f\in L_p^{\textrm{loc}}(\R^n)$ and
	\[
	\| f\|_{M_p^\lambda}  = \sup_{r>0}r^{-\lambda}\sup_{x\in \R^n}\| f\|_{L^p(B_r(x))},
	\]
	where $B_r(x)$ is the ball of radius $r$ centered at $x\in \R^n$.

For $f\in L_1(\R^n)$, we define the Fourier transform 
	\[
	\widehat{f}(y)= \int_{\R^n} f(x)e^{-2\pi i(x,y)}\, dx.
	\]

	Our main goal
 is to establish sufficient conditions for a function $f$ so that its Fourier transform  belongs to  Morrey or Campanato-type spaces.
 Let us mention that
the theory of  Fourier inequalities with general weights
  has been extensively   developed since 1980s, see, e.g., \cite{BH} and the references therein as well as the recent papers
  \cite{rev, sin}.
  
In what follows, considering Fourier norm inequalities,  we will assume that $f$ is integrable, since in this case, by a standard argument, there exists an extension of the integral transform  to a linear operator on the whole corresponding  spaces.

We start with the following simple observation concerning the Morrey space. Combining the known embeddings between Morrey and Lebesgue spaces (see, e.g., \cite[Theorem~2.2]{RSS}) and the Hausdorff-Young inequality \cite{hunt, SW}, one immediately arrives at
\begin{equation}
\label{EQMorrey-Lorentzembedding}
\|\widehat f\|_{M_p^\lambda}
\lesssim\|\widehat f\|_{L_s}\lesssim \|f\|_{L_{s',s}}
\lesssim \| f\|_{L_{s'}}, \qquad s\geq 2,
\end{equation}
where $\frac{1}{s} = \frac{1}{p}-\frac{\lambda}{n}$, $s'=\frac{s}{s-1}$, and $L_{s',s}=L_{s',s}(\R^n)$ denotes the classical Lorentz space. Moreover, by homogeneity, the relation $\frac{1}{s} = \frac{1}{p}-\frac{\lambda}{n}$ is necessary for inequality \eqref{EQMorrey-Lorentzembedding} to hold.

We will see that the embedding   between Morrey and Lorentz spaces (see Lemma~\ref{LEMmorrey-lorentzembedding} below), rather than Morrey--Lebesgue embeddings, implies a sharper version of \eqref{EQMorrey-Lorentzembedding}, namely,
\begin{equation}
	\label{EQmorrey-weaklqineq}
	\|\widehat f\|_{M_p^\lambda}\lesssim\|\widehat f\|_{L_{s,\infty}}\lesssim \|f\|_{L_{s',\infty}}, \qquad s>2,
\end{equation}	
with $\frac{1}{s} = \frac{1}{p}-\frac{\lambda}{n}$. One of our main results (Theorem~\ref{THMmorreynoweight}) states that the inequality $\|\widehat f\|_{M_p^\lambda}\lesssim \|f\|_{L_{s',\infty}}$ can be substantially improved. In particular, we prove that, for $0<\lambda \leq \frac{n}{p}$, $p\geq 2$,
\[
	\| \widehat{f}\|_{M_{p}^\lambda}\lesssim \sup_{m\in \Z}\sup_{e\subset \Z^n} \frac{(1+\ln |e|)^{n+1}}{|e|^\frac{\lambda}{n}}\frac{1}{\big|\cup_{k\in e} Q_k^m \big|^{\frac{1}{p}-\frac{\lambda}{n}}}\int_{\cup_{k\in e} Q_k^m} |f(x)|\, dx,
\]
where $Q_{k}^m = [0,2^m)^n+2^m k$, $k\in \Z^n$ and $m\in \Z$. In Corollary~\ref{Cormorreynoweight} we will see that the latter sharpens \eqref{EQmorrey-weaklqineq}.

Importantly,  \eqref{EQmorrey-weaklqineq} can be extended for the whole range of the Lorentz spaces, namely,
\[
\| \widehat{f}\|_{M_{p,q}^\lambda}\lesssim \| {f}\|_{L_{s',q}},
\]
where $0<\frac{1}{s}=\frac{1}{p}-\frac{\lambda}{n}<\frac{1}{2}$ and
$0<q\le \infty$. Here $M_{p,q}^\lambda$ is the three parameter Morrey space,
see the precise definitions in the next section.

We would like to mention that  Fourier inequalities between two Morrey norms are not valid in general, that is, if $\lambda>0$, there is no constant $C$ such that
$
\| \widehat{f} \|_{M^{\nu}_q(\R^n)}\le C\|f \|_{M^{\lambda}_p(\R^n)}
$
for every $f\in M^{\lambda}_p(\R^n)\cap L_1(\R^n)$; see Appendix A.
Replacing the Morrey norm by the (weighted) Lebesgue or Lorentz norm in the left-hand side, the inequality becomes meaningful, cf. \eqref{EQMorrey-Lorentzembedding} and \eqref{EQmorrey-weaklqineq}.

Furthermore, using embeddings  between Morrey and Lebesgue spaces and the known Fourier  inequalities in   Lebesgue spaces we will show that the weighted inequality
\[
	\| |x|^{-\delta} \widehat{f}\|_{M_{q}^{\nu}}\lesssim \| |x|^\gamma f\|_{L_p}
\]
 holds  if and only if
	\[
	\max\bigg\{ 0, n\bigg(\frac{1}{q} -\frac{1}{p'} \bigg)-\nu \bigg\}\leq \delta<\frac{n}{q}-\nu, \qquad
	\gamma =\delta+n\bigg(\frac{1}{p'}-\frac{1}{q}\bigg)+\nu.
	\]
This extends  the classical Pitt inequality
\[
	\| |x|^{-\delta} \widehat{f}\|_{L_{s}}\lesssim \| |x|^\gamma f\|_{L_p},\qquad 1<p\leq s<\infty,
\]
provided
$		\max\big\{ 0, n\big( \frac{1}{s}-\frac{1}{p'} \big)  \big\}\leq \delta<\frac{n}{s},$
see, e.g., \cite{BH}.

We also consider the Campanato-type spaces $C_{p,q}^{\lambda}$. These spaces have an additional integrability parameter $q$, so that $C_{p,\infty}^{\lambda}$ corresponds to the classical Campanato  space \eqref{camp+}.
In Section 3.3,
we will  see   that  $C_{p,q}^{\lambda}$ coincides (and the corresponding norms are equivalent) either  with the Morrey space $M_{p,q}^{\lambda}$ when
$0< \lambda< \frac{n}{p}$
or with the Besov space $B_{\infty,q}^{\lambda-\frac np}$ when $\frac{n}{p}< \lambda< \frac{n}{p}+1$.
For $p\geq 1$, in the case $q=\infty$ this  recovers the classical result that
the Campanato space (defined over bounded domains $\Omega$ satisfying certain conditions) generalizes the scale of Morrey spaces. More precisely,
	\begin{equation}
	\label{EQcampanatocharact}
	{{C}}_p^\lambda(\Omega)  =\begin{cases}
	M_p^\lambda(\Omega),&\text{if }0\leq \lambda <\frac{n}{p},\\
	BMO(\Omega),&\text{if }\lambda=\frac{n}{p},\\
	\textrm{Lip } \alpha (\overline{\Omega}),&\text{if }\lambda = \frac{n}{p}+\alpha, \, 0<\alpha\leq 1,
	\end{cases}
	\end{equation}
 see \cite[\S 1.7.2]{triebelII} as well as the pioneering  papers by  Campanato  \cite{Ca} for $\lambda\neq\frac{n}{p}$ and John and Nirenberg \cite{JN}
 for $\lambda=\frac{n}{p}$. In the case $\lambda>\frac{n}{p}+1$ the  space ${{C}}_p^\lambda(\Omega)$ is trivial.

The Fourier inequalities we obtain in the Campanato spaces $C_{p,q}^{\lambda}$ with $\frac{n}{p}< \lambda< \frac{n}{p}+1$ differ from those in the Morrey spaces. We claim that
\begin{equation}
	\label{hhh}
	 | \widehat{f} |_{C_{p,q}^\lambda}\lesssim
 \|  f\|_{
 T^{\lambda-\frac{n}{p}}_{q}L_1},
\end{equation}
where the right-hand side is the norm in the truncated Lebesgue space, see \eqref{localtruncated}. In particular, these yield corresponding inequalities for Besov spaces, for instance,
$
  |\widehat{f}|_{B_{\infty,1}^\alpha}
\lesssim   \big\||x|^\alpha f\big\|_1
$, $0<\alpha<1$,
cf. Remark~\ref{THMcampanato'}.

The outline of the paper is as follows.
In Section~\ref{SECfunctionspacesdef} we introduce the Morrey and Campanato-type  spaces.  In Section~\ref{SECfunctionspacesprop} we study useful properties of these spaces in detail, such as norm discretization, embeddings, and interpolation.

    Section~\ref{SECfouriermorrey} contains   inequalities for the Fourier transform in  Morrey-type spaces (both in the weighted and unweighted settings).
       In Section~\ref{section5} we study Pitt's inequalities between Morrey and Lebesgue spaces. In particular, we find necessary and sufficient conditions
for $		\| |x|^{-\delta} \widehat{f}\|_{M_{q}^{\nu}}\lesssim \| |x|^\gamma f\|_{L_p}
$,
provided that
$1<p<\infty$, $0<q<\infty$, and
$0<\frac{1}{q}-\frac{\nu}{n}<\frac{1}{p}.$

In   Section~\ref{SECfouriercampanato} we prove an analogue of inequality  \eqref{hhh} for weighted Campanato-type spaces. We conclude with Appendix~\ref{SECnofourierineq}, where we show that Fourier norm inequalities  between two Morrey spaces are not possible.

Here and in what follows, the symbols $A\lesssim B  $ and $A\gtrsim B$ stand for the inequalities $A\leq CB$ and $A\geq CB$, respectively, where $C$ is a constant not depending of essential quantities.
By  $\chi_E$ we denote the characteristic function of a set $E$.

\vspace{0.6mm}	\section{Function spaces}	\label{SECfunctionspacesdef}

	\subsection{Morrey and Campanato-type spaces}

	For $0<p<\infty$, $0<q\leq \infty$, and $0\le \lambda\leq \frac{n}{p}$, the (global) Morrey-type space $M^{\lambda}_{p,q}=M^{\lambda}_{p,q}(\R^n)$ is defined to be the space of all functions $f\in L_p^{\textrm{loc}}(\R^n)$ such that
	\[
	\| f\|_{M^{\lambda}_{p,q}}= \begin{cases} \displaystyle \bigg( \int_0^\infty \bigg( r^{-\lambda}\sup_{x\in \R^n} \|f\|_{L_p(B_r(x))}\,\bigg)^q\frac{dr}{r} \bigg)^\frac{1}{q}<\infty, &\text{if }q<\infty,\\
	\displaystyle \sup_{r>0}r^{-\lambda} \sup_{x\in \R^n}\| f\|_{L_p(B_r(x))}<\infty, &\text{if }q=\infty.
	\end{cases}
	\]
	Note that $M^\lambda_{p,\infty}= M^\lambda_p$. We also observe that for $q<\infty$, the space $M_{p,q}^{0}$ is trivial, and by the Lebesgue differentiation theorem, so is the space $M_{p,q}^{\frac{n}{p}}$. For $\lambda>\frac{n}{p}$, all the spaces $M_{p,q}^\lambda$ are trivial. For $q<\infty$, similar spaces with the  supremum outside of the integral were considered in \cite{Bu1}.

To define the local Morrey space, we proceed similarly but instead of considering the supremum on $x\in\R^n$, the center of the ball $x_0\in \R^n$ is fixed. In more detail, for $0<p<\infty$, $0<q\leq \infty$, and $\lambda\geq 0$, the local Morrey space $LM^{\lambda}_{p,q}=LM^{\lambda}_{p,q}(\R^n)$ (centered at the origin) is defined as the space of all functions $f\in L_p^{\textrm{loc}}(\R^n)$ such that
	\[
	\| f\|_{LM^{\lambda}_{p,q}} = \begin{cases} \displaystyle \bigg( \int_0^\infty \Big( r^{-\lambda} \|f\|_{L_p(B_r(0))}\Big)^q\frac{dr}{r} \bigg)^\frac{1}{q}, &\text{if }q<\infty,\\
	\displaystyle	\sup_{r>0} r^{-\lambda}\| f\|_{L_p(B_r(0))}, &\text{if }q=\infty.
	\end{cases}
	\]
 We stress that the space $LM_{p,q}^\lambda$ is not trivial if $\lambda>\frac{n}{p}$ (here we correct \cite[Example 23, (3)]{SDFH}), in contrast with the space $M_{p,q}^\lambda$, where we require $0\leq \lambda\leq \frac{n}{p}$. For example, any function $f\in L_p^{\textrm{loc}}(\R^n)$ that is compactly supported away from the origin is such that  $f\in LM_{p,q}^\lambda$ for any $0<\lambda$ and $0<q\leq \infty$. We also observe that $LM_{p,\infty}^0=L_p$. On the other hand, if $q<\infty$ and $\lambda=0$, the space $LM^\lambda_{p,q}$ is trivial.
In the case $\lambda>0$, the local Morrey space coincides with the  truncated Lebesgue space  $T^{-\lambda}_{q}L_p$, which is of amalgam type (see Lemma~\ref{1} below and
 \cite{BCN, NS}).
	
 For $0<p<\infty$, $0<q\leq \infty$, and  $\lambda \in \R$, the truncated Lebesgue spaces $T^{\lambda}_{q}L_p=T^{\lambda}_{q}L_p({\Bbb \R^n})$
 is defined as the space
  of measurable functions $f\in L_p^{\textrm{loc}}(\R^n)$ such that
\begin{equation}
	\label{localtruncated}
	\| f\|_{T^{\lambda}_{q}L_p} = \begin{cases} \displaystyle \bigg( \sum_{k\in \Z} \bigg( 2^{\lambda k} \|f\|_{L_p\big( B_{2^{k+1}}(0)\backslash B_{2^k}(0)\big)} \bigg)^q \bigg)^\frac{1}{q}, &\text{if }q<\infty,\\
		\displaystyle	\sup_{k\in \Z} 2^{\lambda k}\| f\|_{L_p\big(B_{2^{k+1}}(0)\backslash B_{2^k}(0)\big)}, &\text{if }q=\infty.
	\end{cases}
\end{equation}
 The theory of truncated spaces has been recently developed in \cite{oscar}.

		For a function $f$ defined on $\R^n$ we introduce its average over a ball of radius $r$ (with the measure denoted by $|B_r|$)  centered at $x_0\in \R^n$ as
	\[
	A_rf(x_0)=\frac{1}{|B_r|}\int_{B_r(x_0)} f(x)\, dx = \frac{1}{|B_r|}\int_{B_r(0)} f(x+x_0)\, dx.
	\]	
		For $0<p<\infty$ and $0\leq \lambda\leq \frac{n}{p}+1$, the classical Campanato space ${{C}}^\lambda_p={{C}}_p^\lambda(\R^n)$ is defined by
\begin{equation}
	\label{camp+}
	| f|_{{{C}}^\lambda_p } = \sup_{r>0} r^{-\lambda}\sup_{\xi\in \R^n} \bigg( \int_{B_r(0)} |f(x+\xi)-A_rf(\xi)|^p\, dx\bigg)^\frac{1}{p}<\infty,
\end{equation}
where the functional $|\cdot |_{{{C}}^\lambda_p}$ is a quasi-seminorm, which becomes a quasi-norm after appropriate modifications, cf.  \eqref{d3} below.
As we stated in   \eqref{EQcampanatocharact}, the Campanato scale  generalizes the scale of Morrey spaces.

In the same spirit as the Morrey-type spaces $M_{p,q}^\lambda$, we define the spaces $C_{p,q}^\lambda$. Let $0<p<\infty$, $0<q\leq \infty$, and $0\leq\lambda\leq\frac{n}{p}+1$. We define  the Campanato space $C_{p,q}^\lambda=C_{p,q}^\lambda({\Bbb \R^n})$ as the set of measurable functions $f\in L_p^{\textrm{loc}}(\R^n)$
such that
	\begin{equation}\label{d3}
		\| f\|_{C_{p,q}^\lambda} =
	| f|_{C_{p,q}^\lambda}+\sup_{x\in\Bbb R^n}\|f\|_{L_p(B_1(x))}<\infty,
	\end{equation}
where
\[
	| f|_{C_{p,q}^\lambda}:=\bigg(\int_0^\infty \bigg(r^{-\lambda}\sup_{x\in \R^n} \bigg(\int_{B_r(x)}|f(y)-A_rf(x)|^p\, dy \bigg)^\frac{1}{p}\bigg)^q\frac{dr}{r}\bigg)^\frac{1}{q},
	\]
with the obvious modification if $q=\infty$ (then we write $C_{p,\infty}^\lambda=C_{p}^\lambda$).
In the case $q<\infty$, these spaces were first introduced in \cite{Sta}.

We will see in Propositions~\ref{lem6} and \ref{lem8}  that for $1\leq p<\infty$, $\| f\|_{C_{p,q}^\lambda}\asymp \| f\|_{M_{p,q}^\lambda}$ whenever $0\leq\lambda<\frac{n}{p}$ under some natural conditions on $f$, and
$\| f\|_{C_{p,q}^\lambda}\asymp \| f\|_{B_{\infty,q}^{\lambda-\frac np}}
$ whenever $\frac{n}{p}<\lambda<\frac{n}{p}+1$, where $B_{\infty,q}^{\lambda-\frac np}$ is the Besov space (see Subsection~\ref{SECbesov} below for the precise definitions).

\subsection{Weighted Morrey and Campanato-type spaces}
To define weighted spaces, we replace  $r^{-\lambda}$ in the definitions of
Morrey and Campanato spaces
by a general decreasing function satisfying some integrability conditions. For $0<p<\infty$,  $0<q\leq \infty$, and $k>0$, we denote by  $\Xi^k_{p,q}$ the class of nonnegative weight functions  $u$ on $\R_+=(0,\infty)$ satisfying
\begin{equation}
	\label{EQxiclass}
	\big\| r^{\frac{k}{p}-\frac{1}{q}}u(r)\big\|_{L_q(0,\varepsilon)} <\infty,\qquad  \big\|r^{-\frac{1}{q}}u(r)\big\|_{{L_q(\varepsilon,\infty)}}<\infty,
\end{equation}
for every fixed $\varepsilon>0$. If $q=\infty$, we simply denote $\Xi_{p,\infty}^k=\Xi_{p}^k$. For $u\in \Xi_{p,q}^n$, the weighted Morrey-type space $M^u_{p,q}=M^u_{p,q}({\Bbb \R^n})$ is then defined by
\[
\| f\|_{M^u_{p,q}} =\begin{cases}
	\displaystyle\bigg(\int_0^\infty \Big( u(r)\sup_{x\in \R^n}\| f\|_{L_p (B_r(x))}\Big)^q\, \frac{dr}{r}\bigg)^\frac{1}{q},&\text{if }q<\infty,\\
	\displaystyle\sup_{r>0}u(r)   \sup_{x\in \R^n}\| f\|_{L_p (B_r(x))},&\text{if }q=\infty.
\end{cases}
\]
It is seen at once that if $u(r)=r^{-\lambda}$, then $M^u_{p,q}=M^\lambda_{p,q}$. Moreover, the assumption $u\in \Xi_{p,q}^n$ is essential, since if $u$ is a decreasing weight not belonging to this class, then the corresponding space $M^u_{p,q}$ is trivial (see \cite[Lemma 4.1]{Bu1} and the references therein; we remark that the definition of weighted Morrey space is slightly different in that paper).

 Let $0<p<\infty$, $0<q\leq \infty$, and  $u\in \Xi_{p,q}^{n+p}$. The weighted Campanato space $C_{p,q}^u=C_{p,q}^u(\R^n)$ is defined by
\[
\| f\|_{C_{p,q}^u} =\bigg(\int_0^\infty \bigg(u(r)\sup_{x\in \R^n} \bigg(\int_{B_r(x)}|f(y)-A_rf(x)|^p\, dy \bigg)^\frac{1}{p}\bigg)^q\frac{dr}{r}\bigg)^\frac{1}{q}+\sup_{x\in\Bbb R^n}\|f\|_{L_p(B_1(x))},
\]
with the usual modification for
$q=\infty$. Here the integral defines the  associated seminorm $|f|_{C_{p,q}^u}$. The assumption $u\in \Xi_{p,q}^{n+p}$ is natural to work with nontrivial Campanato spaces (see Remark~\ref{REMcampanatotrivial} below).

\subsection{Besov spaces}\label{SECbesov}
 Let $1\leq p\leq\infty$, $0<q\leq \infty$, and $0<\alpha < 1$. Define the Besov space as
 	\[
 	B_{p,q}^\alpha= 	B_{p,q}^\alpha(\Bbb R^n)=\bigg\{f\in L_p(\Bbb R^n):\| f\|_{B_{p,q}^\alpha}:=\bigg( \int_0^1( t^{-\alpha} \omega(f,t)_p)^q\, \frac{dt}t\bigg)^\frac{1}{q}+\|f\|_{L_p(\Bbb R^n)}<\infty\bigg\}
 	\]
	(with the obvious modification if $q=\infty$). Here $\omega(f,\cdot)_p$ denotes the $p$-modulus of continuity  of $f\in L_p(\Bbb R^n)$ given by
 	\[
 	\omega(f,t)_p=\sup_{|y|\leq t}\|f(\cdot+y)- f(\cdot)\|_{L_p}.
 	\]
	Note that $B_{\infty,\infty}^\alpha=\text{Lip}\, \alpha$.
It is well known that
the Fourier-analytically defined Besov norm is equivalent to $\| f\|_{B_{p,q}^\alpha}$, see \cite{triebelII}.
By $|f|_{B_{p,q}^\alpha}$ we denote the seminorm
 	$
 	|f|_{B_{p,q}^\alpha}:= \big( \int_0^1( t^{-\alpha} \omega(f,t)_p)^q\,  \frac{dt}t\big)^\frac{1}{q}.
 	$

\section{Properties of Morrey and Campanato-type spaces}\label{SECfunctionspacesprop}

 \subsection{Hardy's inequalities}
We begin with the following version of discrete Hardy's inequalities, which will be systematically used in the sequel. The proof can be given modifying \textnormal{\cite[Lemma 2.5]{PST}}.

\begin{lemma}
	\label{LEMhardyseries}
	Let $0< p\le\infty$ and $\{a_n\}_{n\in \Z}$, $\{b_n\}_{n\in \Z}$ be nonnegative sequences.
	\begin{enumerate}
		\item If $\sum_{k=-\infty}^n b_k\lesssim b_n$ for every $n\in \Z$, then
		\[
		\sum_{n\in \Z} b_n^p a_n^p \leq \sum_{n\in \Z}  \bigg(b_n \sum_{k=n}^\infty a_k\bigg)^p\lesssim\sum_{n\in \Z} b_n^p a_n^p.
		\]
		\item If $\sum_{k=n}^\infty b_k \lesssim b_n$ for every $n\in \Z$, then
		\[
		\sum_{n\in \Z} b_n^p a_n^p \leq \sum_{n\in \Z}  \bigg( b_n\sum_{k=-\infty}^n a_k\bigg)^p\lesssim\sum_{n\in \Z} b_n^p a_n^p.
		\]
	\end{enumerate}
\end{lemma}

 \subsection{Basic facts about Morrey and Campanato spaces}
  We begin with two  simple embedding results.
 \begin{lemma}\label{LEMmorrey-lorentzembedding}
 	Let $0<p<\infty$, $0<\lambda\leq \frac{n}{p}$, and $\frac{1}{q}=\frac{1}{p}-\frac{\lambda}{n}$. There holds $L_q\hookrightarrow L_{q,\infty}\hookrightarrow M^\lambda_p$.
 \end{lemma}
 \begin{proof}
    Let $f^*$ be the decreasing rearrangement of $f$. Denote by $|B_r|$ the volume of any ball of radius $r$ on $\R^n$. Since  $\lambda>0$, then $p<q$, and hence for any $r>0$ and $x\in \R^n$ we have
 	\begin{align*}
 		r^{-\lambda} \| f\|_{L_p(B_r(x))} &\leq  r^{-\lambda} \bigg( \int_0^{|B_r|} f^*(t)^p\, dt\bigg)^\frac{1}{p} \leq  r^{-\lambda}\sup_{s>0} s^\frac{1}{q} f^*(s) \bigg( \int_0^{|B_r|} t^{-\frac{p}{q}}\, dt\bigg)^\frac{1}{p}\\
 		& \asymp r^{-\lambda+ n\big( \frac{1}{p}-\frac{1}{q}\big)} \| f\|_{L_{q,\infty(\R^n)}} = \| f\|_{L_{q,\infty(\R^n)}}.\qedhere
 	\end{align*}
 \end{proof}

 \begin{lemma}\label{LEMembeddingsmorrey}
 	Let $X\in\{M,LM,C\}$.
 	\begin{enumerate}
 		\item Let $0< p_0<p_1<\infty$ and $0<q\leq \infty$. Then,
 		\begin{align*}
 			X^{\lambda_1}_{p_1,q}&\hookrightarrow X^{\lambda_0}_{p_0,q},
 		\end{align*}
 		whenever $\lambda_1-\frac{n}{p_1}= \lambda_0-\frac{n}{p_0}\geq 0$.
 		\item Let $0<q_0<q_1\leq \infty$. Then,
 		\begin{align*}
 			X^{\lambda}_{p,q_0}&\hookrightarrow X^{\lambda}_{p,q_1}.
 		\end{align*}
 	\end{enumerate}
Similar results hold for $X^{\lambda}_{p,q}=T^{\lambda}_{q}L_p$. \end{lemma}
 The proof is immediate: the first part follows from  H\"older's inequality and   the second part  from Lemma~\ref{lem3.7} below and Jensen's inequality.
Note that, formally, the second part of Lemma \ref{LEMembeddingsmorrey} is valid for arbitrary $\lambda$, although in some cases only trivial spaces are considered.

Lemma ~\ref{LEMembeddingsmorrey} in particular yields that
for $0<p,q<\infty$ and  $\lambda >\frac{n}{p}+1$ or $\lambda<0$, the space $C_{p,q}^\lambda$ consists only of constant functions; cf. Remark \ref{REMcampanatotrivial}.
	Indeed, we have  $C_{p,q}^\lambda  \hookrightarrow C_p^\lambda$ and the space $C_p^\lambda$ consists only of constant functions for $\lambda >\frac{n}{p}+1$ or $\lambda<0$.

We continue with a characterization of local Morrey spaces.

\begin{lemma} \label{1}
	Let $0< p < \infty$, $0<q\leq \infty$,  and $0 < \lambda < \infty$. Then $\DLM=LM_{p, q}^{\lambda}$.
\end{lemma}

\begin{proof}
	On the one hand, by Lemma~\ref{lem3.7} below it is clear that $\| f\|_{\DLM}\lesssim \|f\|_{LM_{p,q}^{\lambda}}$ for any $0<q\leq \infty$. On the other hand, we have
	\[
	\|f\|_{L_p (B_{2^k}(0))}=\bigg(\sum\limits_{m=-\infty}^{k-1}\int_{B_{2^{k+1}}(0)\backslash B_{2^k}(0)} |f(y)|^p\, dy\bigg)^\frac{1}{p}.
	\]
	Consequently, if $q<\infty$,
	\begin{align*}
	\|f\|_{LM_{p,q}^{\lambda}}^q&\asymp  \sum_{k\in \Z}  2^{-k\lambda q} \bigg( \int_{B_{2^k}(0)}|f(y)|^p\, dy \bigg)^{\frac qp}\\& \leq \sum_{k\in \mathbb{Z}}  2^{-k\lambda q}\bigg(\sum\limits_{m=-\infty}^{k-1}\int_{B_{2^{k+1}}(0)\backslash B_{2^k}(0)} |f(y)|^p\, dy\bigg)^{\frac qp}.
	\end{align*}
	Hence, Lemma \ref{LEMhardyseries} and \eqref{EQlm}
     yield
	\[
	\|f\|_{LM_{p,q}^{\lambda}}\lesssim \|f\|_{\DLM}.
	\]
	For $q=\infty$, we have
	\begin{equation*}
	\| f\|_{LM_{p,\infty}^\lambda}\lesssim \sup_{k\in \Z} 2^{-\lambda k }\bigg( \sum_{m=-\infty}^{k-1}2^{\lambda mp}\bigg)^\frac{1}{p}\| f\|_{\DLMinfty} \asymp  \| f\|_{\DLMinfty}.
	\end{equation*}
\end{proof}

\begin{lemma}\label{lem5}
Let $1\leq p<\infty$, $0\leq \lambda\leq \frac{n}{p}+1$, and $0<q\leq\infty$. Then,
\[
| f|_{C_{p,q}^\lambda}
\asymp \bigg( \int_0^\infty\bigg(r^{-\lambda}\sup_{x\in\Bbb R^n}\inf_{c}\|f-c\|_{L_{p}(B_r(x))}\bigg)^q\frac {dr}r\bigg)^{\frac1q}.
\]
\end{lemma}
	The proof is straightforward by observing that
	\begin{equation}
		\label{EQequivcampanatoinf}
			\inf_{c} \|f-c\|_{L_p(B_{r}(x))}\leq \|f- A_{r}f(x)\|_{L_p(B_{r}(x))}\leq  2\inf_{c} \|f-c\|_{L_p(B_{r}(x))},
	\end{equation}
	for any $r>0$.

Next we study  a discretization of the (quasi- or semi-)norms in Morrey and Campanato-type spaces.

\begin{lemma}\label{lem3.7}Let   $0< p<\infty$ and  $0<q\leq\infty$.
	 If $0\leq \lambda< \infty,$ then
	\begin{equation}\label{gg-}
		\|f\|_{M^{\lambda}_{p,q}}\asymp\bigg( \sum_{k\in \mathbb{Z}} \bigg( 2^{-k\lambda} \sup_{x\in \Bbb R^n} \|f\|_{L_{p}(B_{2^k}(x))} \bigg)^q \bigg)^{\frac{1}{q}},
	\end{equation}
	\begin{equation}\label{EQlm}
		\|f\|_{LM^{\lambda}_{p,q}}\asymp\bigg( \sum_{k\in \mathbb{Z}} \bigg( 2^{-k\lambda} \|f\|_{L_{p}(B_{2^k}(0))} \bigg)^q \bigg)^{\frac{1}{q}},
	\end{equation}
	and, for $1\le p<\infty$,
	\begin{equation}\label{gg}
		|f|_{C^{\lambda}_{p,q}}\asymp\bigg( \sum_{k\in \mathbb{Z}} \left( 2^{-k\lambda} \sup_{x\in \Bbb R^n} \|f-A_{2^k}f(x)\|_{L_{p}(B_{2^k}(x))} \right)^q \bigg)^{\frac{1}{q}},
	\end{equation}
with the obvious modification if $q=\infty$.
\end{lemma}
\begin{proof}
	Firstly, equivalences \eqref{gg-} and \eqref{EQlm} are clear. Indeed, the inequalities $\lesssim$ are trivial, and so are the reverse inequalities for $q=\infty$. For $q<\infty$, the reverse inequalities follow from Lemma~\ref{LEMhardyseries}.

Secondly, by Lemma~\ref{lem5}, if $q<\infty$,
	\[
	|f|_{C_{p,q}^{\lambda}}\asymp \left( \int_0^\infty\left(r^{-\lambda}\sup_{x\in\Bbb R^n}\inf_{c}\|f-c\|_{L_{p}(B_r(x))}\right)^q\frac {dr}r\right)^{\frac1q}.
	\]
	By the monotonicity of  $\displaystyle\sup_{x\in\Bbb R^n}\inf_{c\in\Bbb R}\|f-c\|_{L_{p}(B_r(x))}$ with respect to $r$,
  we obtain
	\[
	|f|_{C_{p,q}^{\lambda}}\asymp \bigg( \sum_{k\in \mathbb{Z}} \bigg( 2^{-k\lambda} \sup_{x\in \Bbb R^n}\inf_{c} \|f-c\|_{L_{p}(B_{2^k}(x))} \bigg)^q \bigg)^{\frac{1}{q}}.
	\]
	Using \eqref{EQequivcampanatoinf} with $r=2^k$, we arrive at \eqref{gg}. The case $q=\infty$ is treated similarly.
\end{proof}

We will need the following notation. Recall that $Q_{k}^m$ is the cube
$
Q_{k}^m = [0,2^m)^n+2^m k,
$
where $k\in \Z^n$ and $m\in \Z$.
Then, for fixed $m$, one has
\[
\bigcup_{k\in \Z^n}Q_{k}^m = \R^n,
\]
and
\[
\big|Q_{k}^m \cap Q_{k'}^m\big|=\begin{cases}
	2^{nm} ,&\text{if } k=k',\\
	0,&\text{otherwise}.
\end{cases}
\]
We say that  a cube $Q$ is of order $m\in \Z$ if $Q= Q^m_{k}$ for some $k.$ The set of all cubes of order $m$ is denoted by $G^m$. Note also that for $r\geq m$, each cube of order $r$ can be subdivided into $2^{n(r-m)}$ disjoint cubes of order $m$.

\begin{remark}\label{REMcubes}
In the  definition of the  Morrey space $M^\lambda_{p,q}$ one can equivalently replace integration over balls $B_r(x)$ by integration over cubes $Q_r(x)$ centered at $x$ with side length $2r$. Moreover, there holds
	\begin{align*}
		\| f\|_{M^\lambda_{p,q}} &\asymp \bigg(\sum_{m\in \Z}\Big(2^{-\lambda m} \sup_{Q\in G^m}\| f\|_{L_p(Q)} \Big)^q\bigg)^\frac{1}{q}.
	\end{align*}
For weighted Morrey-type spaces $M_{p,q}^u$, the same observation is valid whenever $u(r)\asymp u(2r)$ for $r>0$, that is,
\[
\| f\|_{M_{p,q}^u}\asymp \bigg(\int_0^\infty \Big(u(r)\sup_{x\in \R}\|f\|_{L_p(Q_r(x))}\Big)^q\, \frac{dr}{r} \bigg)^\frac{1}{q}\asymp \bigg(\sum_{m\in\Z} \Big(u(2^m)\sup_{Q\in G^m}\| f\|_{L_p(Q)} \Big)^q\bigg)^\frac{1}{q}.
\]
\end{remark}

One more observation concerning weighted Campanato spaces is in order.
\begin{remark}\label{REMcampanatotrivial}
	We note that in the definition of the Campanato space $C_{p,q}^u$, in the case $p\geq 1$, the condition $u\in \Xi_{p,q}^{n+p}$ is necessary for the corresponding space to be nontrivial (i.e., not to contain only constant functions).
	
	Assume that the first condition in \eqref{EQxiclass} does not hold, i.e., for some $\varepsilon>0$ we have $\big\| u(r) r^{\frac{n}{p}+1-\frac{1}{q}}\big\|_{L_q(0,\varepsilon)}=\infty$. In this case, if $f\in C_{p,q}^u$, then
 \[
 \liminf_{r\to 0} \sup_{x\in \R^n}r^{-\frac{n}{p}-1}\| f-A_{r}f(x)\|_{L_p(B_{r}(x))}=0.
 \]
 Then we follow the proof\footnote{The proof of \cite[Theorem 1.1]{Slavin} was given for $p=2$ but it can be extended to any $p\ge1$.} of Theorem 1.1 in \cite{Slavin} to obtain that $f$ is a  constant.

 On the other hand, assume that the second condition of \eqref{EQxiclass} does not hold. Since $\|f-A_{r}f(x)\|_{L_{p}(B_r(x))}\geq \inf_c\| f-c\|_{L_p(B_r(x))}$ and the latter is increasing in $r$, we get that
	$
	\|f-A_{r}f(x)\|_{L_{p}(B_r(x))}>M_f>0,
	$
	for sufficiently large $r$, unless $f$ is a constant. Thus, if $f$ is not  a constant, then $|f|_{C_{p,q}^u}\geq M_f \big\| r^{-\frac{1}{q}} u(r)\big\|_{L_q(\varepsilon,\infty)}=\infty$ for $\varepsilon$ sufficiently large.	
\end{remark}

\subsection{When  the Campanato space is either Morrey or Besov}
As mentioned in \eqref{EQcampanatocharact}, the classical Campanato space $C_p^\lambda$
 extends the notion of functions of bounded mean oscillation as well as  the Morrey and  H\"{o}lder  spaces.
 It is natural to ask whether a similar extension holds for $C_{p,q}^{\lambda}$ spaces. The answer to this question is given in Propositions \ref{lem6} and \ref{lem8}.
\begin{proposition}\label{lem6}
	Let $1\leq p<\infty$, $0< \lambda< \frac{n}{p}$, and $0<q\leq\infty$.
If a measurable function $f$ is such that,
 for every $x\in \R^n$,
	\begin{equation}\label{d3-}
		\lim_{r\to+\infty}\frac1{|B_r|}\int_{B_r(x)}f(y) \,  dy=0,
	\end{equation}
then
 		\[
 		\|f\|_{M_{p,q}^{\lambda}}\asymp\|f\|_{C_{p,q}^{\lambda}}.
 		\]

\end{proposition}
\begin{proof}
On the one hand, Lemma  \ref{lem5} yields
  $
  \|f\|_{C_{p,q}^{\lambda}}\lesssim  \|f\|_{M_{p,q}^{\lambda}}$ for any $q$. On the other hand, for $f\in C_{p,q}^{\lambda}$ and $q<\infty$, we have, by Lemma~\ref{lem3.7},
 		\begin{align*}
 		\| f\|_{M_{p,q}^\lambda}& \asymp 	\displaystyle\bigg(\sum_{k\in \Z} \bigg(2^{-\lambda k}\sup_{x\in \R^n} \|f\|_{L_p(B_{2^k}(x))}\bigg)^q\bigg)^\frac{1}{q}\\
 		& \lesssim \displaystyle\bigg(\sum_{k\in \Z} \bigg(2^{-\lambda k}\sup_{x\in \R^n} \|f-A_{2^k}f(x)\|_{L_p(B_{2^k}(x))}+2^{ (\frac{n}p-\lambda)k}\sup_{x\in \R^n} |A_{2^k}f(x)|\bigg)^q\bigg)^\frac{1}{q}\\
		& \leq  |f|_{C_{p,q}^{\lambda}}+\displaystyle\bigg(\sum_{k\in \Z} \bigg(2^{(\frac np-\lambda )k}\sup_{x\in \R^n} |A_{2^{k}}f(x)| \bigg)^q\bigg)^\frac{1}{q}.
 		\end{align*}
 		Condition \eqref{d3-} implies
 		\[
 		\sup_{x\in \R^n} |A_{2^k}f(x)|=\sup_{x\in \R^n}\bigg|\sum_{m=k}^\infty A_{2^m}f(x)- A_{2^{m+1}}f(x)\bigg|\le \sum_{m=k}^\infty \sup_{x\in \R^n}|A_{2^m}f(x)- A_{2^{m+1}}f(x)|.
 		\]
 		By Hardy's inequality (Lemma~\ref{LEMhardyseries}),
 		\begin{align*}
 		\|f\|_{M_{p,q}^\lambda}&\lesssim|f|_{C_{p,q}^{\lambda}}+\displaystyle\bigg(\sum_{k\in \Z} \bigg(2^{(\frac np-\lambda )k}\sum_{m=k}^\infty \sup_{x\in \R^n}|A_{2^m}f(x)- A_{2^{m+1}}f(x)| \bigg)^q\bigg)^\frac{1}{q}\\
 		&\leq |f|_{C_{p,q}^{\lambda}}+\bigg(\sum_{k\in \Z} \bigg(2^{(\frac np-\lambda )k}\sup_{x\in \R^n}| A_{2^{k}}f(x) -A_{2^{k+1}}f(x)| \bigg)^q\bigg)^\frac{1}{q}.
 		\end{align*}
 		Further,
 		\begin{align}
 		| A_{2^{k}}f(x) -A_{2^{k+1}}f(x)|&=\bigg(\frac1{|B_{2^{k}}|}\int_{B_{2^{k}}(x)}| A_{2^{k}}f(x) -A_{2^{k+1}}f(x)|^p\, dy\bigg)^{\frac1p}\nonumber \\
 		&\lesssim 2^{-\frac {nk}p}\big(\|f-A_{2^{k}}f(x)\|_{L_p(B_{2^{k}}(x))}+
 		\|f-A_{2^{k+1}}f(x)\|_{L_p(B_{2^{k+1}}(x))}\big).\label{d7}
 		\end{align}
 		Combining the above estimates, we get
 		\[
 		\|f\|_{M_{p,q}^{\lambda}}\lesssim  |f|_{C_{p,q}^{\lambda}}.
 		\]
		The case $q=\infty$ is treated similarly.
 	\end{proof}
 		
 	Our next result studies the interpolation properties of the Campanato spaces $C_{p,q}^\lambda$. We will use it later to characterize Campanato norms in the case $\frac{n}{p}<\lambda<\frac{n}{p}+1$ in terms of the Besov norms.
 	\begin{lemma}\label{T1}
 		Let  $0< p <\infty$, $\theta \in (0,1)$, and $0<q\leq\infty$. Let also $0<\lambda_0 <\lambda_1\le \frac np+1$. Then,
 		\begin{equation*}\label{12}
 			\big( C_{p,\infty}^{\lambda_0},  C_{p,\infty}^{\lambda_1} \big)_{\theta,q} \hookrightarrow   C_{p,q}^{\lambda},
 		\end{equation*}
 		where $\lambda =(1-\theta)\lambda_0 + \theta \lambda_1.$
 	\end{lemma}
 	
 	For $q=\infty$ and $\lambda_i<\frac np$, see \cite[Lemma 2.87 (iii)]{si}.
 	
 	\begin{proof}[Proof of Lemma~\ref{T1}]
Let $x\in \Bbb R^n$ and $f\in \big(C^{\lambda_0}_{p,\infty},C^{\lambda_1}_{p,\infty}\big)_{\theta,q}$.
Using the representation $f=f_0+f_1$ with $f_0\in C^{\lambda_0}_{p,\infty}$ and $f_1 \in C^{\lambda_1}_{p,\infty}$, we write, for $r>0$,
\begin{align*}
	&\phantom{=}r^{-\lambda}\|f-A_{r}f(x)\|_{L_{p}(B_{r}(x))}
	\\
	&\lesssim r^{-(\lambda-\lambda_0)} \Big( r^{-\lambda_0 }\|f_0-A_{r}f_0(x)\|_{L_{p}(B_{r}(x))}+ r^{\lambda_1-\lambda_0}r^{-\lambda_1}\|f_1-A_{r}f_1(x)\|_{L_{p}(B_{r}(x))}\Big)
	\\
	&\lesssim r^{-\theta(\lambda_1-\lambda_0)} \Big(\|f_0\|_{C^{\lambda_0}_{p,\infty}} +r^{\lambda_1-\lambda_0}\|f_1\|_{C^{\lambda_1}_{p,\infty}} \Big).
\end{align*}	
This implies  that for any $r>0$,
\[
\sup_{x\in \Bbb R^n}r^{-\lambda }\|f-A_{r}f(x)\|_{L_{p}(B_{r}(x))}\lesssim
r^{-\theta(\lambda_1-\lambda_0)} K\big(r^{\lambda_1-\lambda_0}, f; C^{\lambda_0}_{p,\infty},  C^{\lambda_1}_{p,\infty}\big),
\]
where $K\big(t, f; X, Y\big)$ denotes the $K$-functional of $f$ for a couple $(X,\,Y)$.
Hence,
\[
|f|_{C^{\lambda}_{p,q}}
\lesssim \bigg(\int_0^\infty \Big( r^{-\theta(\lambda_1-\lambda_0)}K\big(r^{\lambda_1-\lambda_0}, f; C^{\lambda_0}_{p,\infty},  C^{\lambda_1}_{p,\infty}\big)\Big)^q\frac{dr}{r} \bigg)^\frac{1}{q}\asymp  \|f\|_{\big( C^{\lambda_0}_{p,\infty},C^{\lambda_1}_{p,\infty}\big)_{\theta,q}}.
\]
Since $\sup_{x\in \Bbb R^n}\|f\|_{L_p(B_1(x))}\lesssim  \|f\|_{\big( C^{\lambda_0}_{p,\infty},C^{\lambda_1}_{p,\infty}\big)_{\theta,q}}$,
the conclusion follows.
\end{proof}
 	
 	Now we study the relationship between Campanato and Besov spaces. Peetre mentioned in \cite{peetre} (without the proof) that   	
 $
C_{p,q}^{\lambda}= 		B_{\infty,q}^{\alpha}
 $   with $\alpha=\lambda-\frac np$. In the next result we obtain a norm equivalence for these spaces.

 	\begin{proposition}
 		\label{lem8}
 		Let $1\leq p<\infty$, $0<q\leq\infty$, and $\frac{n}{p}< \lambda< \frac{n}{p}+1$.  Then we have
 		\[
 		\|f\|_{C_{p,q}^{\lambda}} \asymp \|f\|_{B_{\infty,q}^{\alpha}},
\qquad   \alpha=\lambda-\frac np.
 		\]

 	\end{proposition}
 	\begin{proof}
 In the case $q=\infty$ it is easy to see  that
 $|f|_{C_{p,\infty}^{\lambda}}
 \lesssim
 |f|_{B_{\infty,\infty}^\alpha}=|f|_{\textnormal{Lip}\, \alpha}$ (see, e.g., \cite[Theorem 5.7.1 (ii)]{PKJF}).
Thus, $\|f\|_{C_{p,\infty}^{\lambda}}
\lesssim
\|f\|_{B_{\infty,\infty}^\alpha}
$
for $\frac{n}{p}< \lambda\le \frac{n}{p}+1$. Taking this into account, as well as the interpolation properties of the Campanato (see Lemma \ref{T1}) and Besov spaces, we obtain the embedding
	\[
B_{\infty,q}^{\alpha}=	\left( B_{\infty,\infty}^{\alpha_0},  B_{\infty,\infty}^{\alpha_1} \right)_{\theta,q}
\hookrightarrow
\big( C_{p,\infty}^{\lambda_0},  C_{p,\infty}^{\lambda_1} \big)_{\theta,q}\hookrightarrow C_{p,q}^{\lambda}.
	\]
Moreover,
$\|f\|_{C_{p,q}^{\lambda}}
\lesssim
\|f\|_{B_{\infty,q}^\alpha}
$
with  $\frac{n}{p}< \lambda< \frac{n}{p}+1$.

	To prove the reverse embedding, let $f\in C_{p,q}^{\lambda}$. We note that
	\[
	|f|_{B_{\infty,q}^\alpha} \asymp \bigg( \sum_{k=-\infty}^0 \big(2^{-k\alpha}\omega(f,2^k)_\infty\big)^q\bigg)^\frac{1}{q}.
	\]
	By the Lebesgue differentiation theorem, we have
	$
	\lim\limits_{m\to -\infty}A_{2^m}f(x)=f(x)
	$ for a.e. $x.$
	Then, for $k\in \Z$, we can write
	$
	A_{2^k}f(x)-f(x)=\sum_{m=-\infty}^kA_{2^m}f(x)-A_{2^{m-1}}f(x).
	$
	Thus,
	\begin{multline*}
	\omega(f,2^k)_\infty=\sup_{\substack{0<|\delta|\leq 2^k\\x\in\Bbb R^n}}
	|f(x+\delta)-f(x)|
\\
	\leq \sup_{\substack{0<|\delta|\leq 2^k\\x\in\Bbb R^n}}\big
	(|f(x+\delta)-A_{2^{k+1}}f(x+\delta)| + |f(x)-A_{2^{k+1}}f(x)|+|A_{2^{k+1}}f(x+\delta)-A_{2^{k+1}}f(x)|\big)
\\
	\leq 2 \sum_{m=-\infty}^{k+1}\sup_{x\in\Bbb R^n}|A_{2^m}f(x)-A_{2^{m-1}}f(x)|+\sup_{\substack{0<|\delta|\leq 2^k\\x\in\Bbb R^n}}|A_{2^{k+1}}f(x+\delta)-A_{2^{k+1}}f(x)|.
	\end{multline*}
	Denote $D_{k,\delta}(x)=B_{2^{k+1}}(x+\delta)\cap B_{2^{k+1}}(x)$. Since for $|\delta| \leq 2^k$ we have $2^{kn}\lesssim |D_{k,\delta}(x)|$, then
	\begin{align*}
	&\phantom{=}|A_{2^{k+1}}f(x+\delta)-A_{2^{k+1}}f(x)|=\bigg(\frac1{|D_{k,\delta}(x)|}\int_{D_{k,\delta}(x)}
	|A_{2^{k+1}}f(x+\delta)-A_{2^{k+1}}f(x)|^p\,dy\bigg)^{\frac1p}\\
	&\leq 2^{-\frac{nk}p}\bigg(\int_{D_{k,\delta}(x)}
	|A_{2^{k+1}}f(x+\delta)-f(y)|^p\,dy\bigg)^{\frac1p}+2^{-\frac{nk}p}\bigg(\int_{D_{k,\delta}(x)}
	|f(y)-A_{2^{k+1}}f(x)|^p\,dy\bigg)^{\frac1p}\\
	&	\leq 2^{-\frac{nk}p}\Big(\|f-A_{2^{k+1}}f(x+\delta)\|_{L_p(B_{2^{k+1}}(x+\delta))}+
	\|f-A_{2^{k+1}}f(x)\|_{L_p(B_{2^{k+1}}(x))}\Big)\\
	&\le  2^{1-\frac{nk}p}\sup_{x\in\Bbb R^n}
	\|f-A_{2^{k+1}}f(x)\|_{L_p(B_{2^{k+1}}(x))}.
	\end{align*}
	Using  \eqref{d7}, we have
	\begin{align*}
	\omega(f,2^k)_\infty&\lesssim \sum_{m=-\infty}^{k+1}2^{-\frac {nm}p}\sup_{x\in\Bbb R^n}\Big(\|f-A_{2^{m}}f(x)\|_{L_p(B_{2^{m}}(x))}+
	\|f-A_{2^{m-1}}f(x)\|_{L_p(B_{2^{m-1}}(x))}\Big)\\
	&+2^{-\frac{nk}p}\sup_{x\in\Bbb R^n}
	\|f-A_{2^{k+1}}f(x)\|_{L_p(B_{2^{k+1}}(x))}\asymp \sum_{m=-\infty}^{k+1}2^{-\frac {nm}p}\|f-A_{2^{m}}f(x)\|_{L_p(B_{2^{m}}(x))}.
	\end{align*}
	Lemmas~\ref{LEMhardyseries}~and~\ref{lem3.7} imply
	\begin{align}\nonumber
		|f|_{B_{\infty,q}^{\alpha}}&\asymp \bigg( \sum_{k=-\infty}^0 \big(2^{-k\alpha}\omega(f,2^k)_\infty \big)^q\bigg)^{\frac1q}
		\\
		&\lesssim \bigg( \sum_{k\in \mathbb{Z}} \Big( 2^{k(-\frac np-\alpha)}\sup_{x\in\Bbb R^n}
		\|f-A_{2^{k}}f(x)\|_{L_p(B_{2^{k}}(x))} \Big)^q\bigg)^{\frac1q}\asymp |f|_{C_{p,q}^{\lambda}}.
\label{ggg}
	\end{align}
	Since
	$
	\lim\limits_{m\to -\infty}A_{2^m}f(x)=f(x),
	$
	we also have
	\[
	|f(x)|\leq\sum_{m=-\infty}^{0}|A_{2^m}f(x)-A_{2^{m-1}}f(x)| +|A_1f(x)|\lesssim \| f\|_{C_{p,q}^\lambda}\bigg(\sum_{m=-\infty}^0 2^{(\lambda-\frac np)q'm}\bigg)^{\frac1{q'}}.
	\]
	Thus, we obtain
$\|f\|_{B_{\infty,q}^\alpha}
\lesssim
\|f\|_{C_{p,q}^{\lambda}},
$ concluding the proof.
  	\end{proof}

 \begin{remark}
 		Let $1\leq p<\infty$, $0<q\leq\infty$, and $\frac{n}{p}< \lambda< \frac{n}{p}+1$.  From the proof of the previous result we see that
$
 |f|_{C_{p,q}^\lambda}+\|f\|_ {L_\infty(\R^n)}$
is an  equivalent norm in  ${C_{p,q}^\lambda}$.

 	\end{remark}

 	Proposition \ref{lem8} together with interpolation properties of Besov spaces imply the following result, which
 is in  sharp contrast with the interpolation properties of  Morrey spaces.

 	\begin{corollary}
 		Let
 $1\leq p_0,  p_1<\infty$,
 $0<q_0, q_1\leq\infty$, $\frac{n}{p_0}<\lambda_0<\frac{n}{p_0}+1$, and $\frac{n}{p_1}<\lambda_1<\frac{n}{p_1}+1$ be such that $\lambda_0<\lambda_1$. For
 		$\theta \in (0,1)$, we have
 		\begin{equation*}\label{12}
 			\left( C_{p_0,q_0}^{\lambda_0},  C_{p_1,q_1}^{\lambda_1} \right)_{\theta,q} =  C_{p,q}^{\lambda},
 		\end{equation*}
 		where
 		$\lambda =(1-\theta)\lambda_0 + \theta \lambda_1$, $\frac 1p=\frac {1-\theta}{p_0} +\frac {\theta}{p_1}$, and $\frac 1q=\frac {1-\theta}{q_0} +\frac {\theta}{q_1}$.
 	\end{corollary}

\vspace{0.4mm}
	\section{Fourier inequalities in Morrey spaces}\label{SECfouriermorrey}

\subsection{Notation and auxiliary lemmas}
	For a sequence $a=\{a_m\}_{m\in \Z^n}$ such that
	\[
	\lim_{\max\limits_j |m_j|\to \infty}a_m=0,
	\]
	we denote by $a^*=\{a_\nu^*\}_{\nu\in \N}$ its decreasing rearrangement. Furthermore, we define the sequence $a^{**}=\{a_\nu^{**}\}_{\nu\in \N}$ as
	\[
	a_n^{**}=\frac{1}{n}\sum_{\nu=1}^n a_\nu^*.
	\]
It is clear that $a_n^*\leq a_n^{**}$. The convolution of two sequences $b=\{b_m\}_{m\in \Z^n}$ and $c=\{c_m\}_{m\in \Z^n}$ is formally defined as
	\[
	(b*c)_m=\sum_{k\in \Z^n} b_kc_{m-k}.
	\]
For $0<p<\infty$ and $0<q\leq\infty$, we define the discrete  Lorentz space  $\ell_{p,q}$ as a collection
of  sequences $a=\{a_m\}_{m\in\Z^n}$ whose (quasi-)norm given by
	\[
	\|a\|_{\ell_{p,q}}:=\begin{cases}
		\bigg(\displaystyle\sum_{\nu=1}^\infty \frac{\big(\nu^{1/p}a^*_\nu\big)^q}{\nu}  \bigg)^{1/q},&\text{if }q<\infty,\\
		\sup_{\nu\in\N}\nu^{1/p}a^*_\nu,&\text{if }q=\infty,
	\end{cases}
	\]
	is finite.
	\begin{lemma}\label{LEMdualnormlorentz}
	Let $1<p,q<\infty$. If $b=\{b_m\}_{m\in \Z^n}$ and $c=\{c_m\}_{m\in \Z^n}$, then
	\[
	\| b*c\|_{\ell_{p,q}} \leq C(p,q) \sup_{\| h\|_{\ell_{p',q'}}=1} \sum_{\nu=1}^\infty \nu b_\nu^{**}c_\nu^{**}h_\nu^{**}.
	\]
	\end{lemma}
	\begin{proof}
	We have, for any $\{h_\nu\}_{\nu\in \Z^n}$,
	\begin{align*}
	\sum_{\nu\in\Z^n}h_\nu(b*c)_\nu& \leq\sum_{k=1}^\infty h_k^*(b*c)_k^*\leq \sum_{k=1}^\infty h_k^*(b*c)_k^{**}=\sum_{k=1}^\infty h_k^*\sup_{|e|=k}\frac1{|e|}\sum_{m\in e}|(b*c)_m|\\
	&\leq \sum_{k=1}^\infty h_k^*\sup_{|e|=k}\frac1{|e|}\sum_{m\in e}\sum_{r\in\Z^n}|b_r c_{m-r}|=\sum_{k=1}^\infty h_k^*\sup_{|e|=k}\sum_{r\in\Z^n}|b_r|\frac1{|e|}\sum_{m\in e}|c_{m-r}|.
	\end{align*}
	We now apply similar estimates in the inner sum,
	\[
	\sum_{r\in\Z^n}|b_r|\frac1{|e|}\sum_{m\in e}|c_{m-r}|\leq\sum_{s=1}^\infty b_s^*\sup_{|\omega|=s}\frac1{|\omega|}\frac1{|e|}\sum_{t\in \omega}\sum_{m\in e}|c_{m-t}|.
	\]
	Thus,
	\[
	\|b*c\|_{\ell_{p,q}}\asymp\sup_{\|h\|_{\ell_{p',q'}}=1}\sum_{\nu\in \Z^n}h_{\nu}(b*c)_{\nu}\leq \sup_{\|h\|_{\ell_{p',q'}}=1}\sum_{k=1}^\infty h_k^*\sum_{s=1}^\infty b_s^* d_{s,k},
	\]
	where
	\[
	d_{s,k}=\sup_{{}^{|\omega|=s}_{|e|=k}}\frac1{|\omega|}\frac1{|e|}\sum_{t\in\omega}\sum_{m\in e}|c_{m-t}|.
	\]
	We now estimate $d_{s,k}$. If $k\geq s$, we have
	\[
	d_{s,k}\leq \sup_{|\omega|=s}\frac1{|\omega|}\sum_{t\in\omega}\sup_{|e|=k}\frac1{|e|}\sum_{m\in e}|c_{m-t}|=\sup_{|\omega|=s}\frac1{|\omega|}\sum_{t\in\omega} c_k^{**}=c_k^{**}.
	\]
	On the other hand, if $k<s$,
	\[
	d_{s,k}\leq \sup_{|e|=k}\frac1{|e|}\sum_{m\in e}\sup_{|\omega|=s}\frac1{|\omega|}\sum_{t\in\omega}|c_{m-t}|=\sup_{|e|=k}\frac1{|e|}\sum_{m\in e}c_s^{**}=c_s^{**}.
	\]
	Consequently,
	\begin{align*}
	\|b*c\|_{\ell_{p,q}}&\lesssim \sup_{\|h\|_{\ell_{p',q'}}=1}\sum_{k=1}^\infty h_k^*\bigg(\sum_{s=1}^k b_s^* d_{s,k}+\sum_{s=k+1}^\infty b_s^* d_{s,k}\bigg)\\
	&\leq\sup_{\|h\|_{\ell_{p',q'}}=1}\sum_{k=1}^\infty h_k^*\bigg(\sum_{s=1}^k b_s^* c_k^{**}+\sum_{s=k+1}^\infty b_s^* c_s^{**}\bigg)\\
	&\leq\sup_{\|h\|_{\ell_{p',q'}}=1}\sum_{k=1}^\infty h_k^*c_k^{**}\sum_{s=1}^k b_s^* +\sum_{s=1}^\infty b_s^* c_s^{**}\sum_{k=1}^{s} h_k^* \leq 2 \sup_{\|h\|_{\ell_{p',q'}}=1}\sum_{k=1}^\infty k h_k^{**}c_k^{**} b_k^{**},
	\end{align*}
as desired.
\end{proof}
	
	\begin{lemma}\label{LEMsuminverses}
		Define the sequence
		\[
		c = \bigg\{c_r= \frac{1}{\prod_{j=1}^n \overline{r}_j}\bigg\}_{r\in \Z^n},
		\]
where $\overline{r}_j= \max\{|{r}_j|,1\}$ for ${r}_j\in \Z$.
		Then the equivalence
		\[
		c_N^{**}= \frac{1}{N}\sum_{\nu=1}^N c_\nu^*\asymp  \frac{\ln^{n} (N+1)}{N}
		\]
		holds.
	\end{lemma}
	
	\begin{proof}
For any $N\in \N$ take  $m\in \N$ such that $2^{m-1}(m-1)^{n-1}\leq N< 2^m m^{n-1}$. Now we define a step hyperbolic cross as follows:
  	\[
  	E_m=\bigcup_{\substack{\nu_1+\cdots+\nu_n\leq m\\ \nu_i>0}}\rho(\nu),
  	\]
	where
	\[
	\rho(\nu)=\rho(\nu_1,\ldots,\nu_n)=\left\{k\in\Z^n: \lfloor 2^{\nu_i-1}\rfloor\leq|k_i|<2^{\nu_i}, \; i=1,2,\ldots,n \right\},
	\]
	and $\lfloor \cdot \rfloor$ denotes the floor function. Note that $|E_m|\asymp 2^m m^{n-1}\asymp N$ (cf. \cite[p. 22]{DTU}).
  Since	the sequence $c$ is nonincreasing in each index, then, by construction,
	\[
	\sum_{\nu=1}^N c_\nu^*\leq\sum_{k\in E_m}|c_k|\leq \sum_{\nu=1}^{|E_m|}c_\nu^*\lesssim \sum_{\nu=1}^N c_\nu^*.
	\]
	Further, we have
\[
\sum_{k\in E_m}|c_k|=\sum_{\substack{\nu_1+\cdots+\nu_n\leq m\\ \nu_i>0}}\sum_{k\in\rho(\nu)}\frac1{\prod_{i=1}^n\bar k_i}\asymp\sum_{\substack{\nu_1+\cdots+\nu_n\leq m\\ \nu_i>0}}(\ln 2)^n\asymp m^{n}.
\]
Consequently, we have
\[
c_N^{**}=\frac1{N}\sum_{\nu=1}^Nc_{\nu}^*\asymp \frac{ m^{n}}{N}\asymp\frac{\ln^{n} (N+1)}{N}.\qedhere
\]
\end{proof}

\subsection{Main result}
Our next result is a  new sufficient condition for the Fourier transform to be in the Morrey space $M_{p,q}^\lambda$.

\begin{theorem}\label{THMmorreynoweight}

	Let $0< p<\infty$, $0<q\leq \infty$, and $\max \big\{0,\frac{n}{p}-\frac{n}{2}\big\} <\lambda< \frac{n}{p}$. Suppose that
	\begin{equation}
		\label{EQfunctionmorrey}
	D_{p,q}^\lambda(f):= \bigg(\sum_{m\in \Z}\bigg( \sup_{e\subset \Z^n} \frac{ (1+\ln |e|)^{n+1}}{|e|^\frac{\beta}{n}}\cdot \frac{1}{\big| \cup_{k\in e} Q_k^{m}\big|^{\frac{1}{p}-\frac{\lambda}{n}}}\int_{\cup_{k\in e} Q_k^{m}}|f(x)|\, dx\bigg)^q\bigg)^\frac{1}{q}<\infty,
	\end{equation}
	where $\beta=\lambda-\max \big\{0,\frac{n}{p}-\frac{n}{2}\big\}$ \textnormal{(}with the usual modification if $q=\infty$\textnormal{)}, then we have
	\begin{equation}
		\label{EQthmmorreybasic}
	\| \widehat{f}\|_{M_{p,q}^\lambda}\lesssim D_{p,q}^\lambda(f).
	\end{equation}
\end{theorem}

\begin{corollary}\label{Cormorreynoweight}
Under the conditions of Theorem \ref{THMmorreynoweight},
we have
\begin{equation}
		\label{Ftr-lorentz}\| \widehat{f}\|_{M_{p,q}^\lambda}\lesssim \| {f}\|_{L_{s',q}}, \qquad 0<q\le \infty,
	\end{equation}
where $0<\frac{1}{s}=\frac{1}{p}-\frac{\lambda}{n}<\frac{1}{2}$. In particular,
\[
\|\widehat f\|_{M_p^\lambda}\lesssim \|f\|_{L_{s',\infty}}, \qquad 2<s<\infty,
\]
cf. \eqref{EQmorrey-weaklqineq}.

Moreover, estimate
\eqref{EQthmmorreybasic}  is sharper than
\eqref{Ftr-lorentz} as for any  $0<q\leq \infty$ there exists $f$ such that $D_{p,q}^\lambda(f)<\infty$ but $\| {f}\|_{L_{s',q}}=\infty$.
\end{corollary}

\begin{remark}
	\begin{enumerate}[label=(\roman{*}),wide = 0pt]
	\item Note that $\beta$ is always positive;	if $p\geq 2$, then $0<\lambda=\beta \leq \frac{n}{p}$ and if $0<p<2$, then $0<\beta\leq \frac{n}{2}$.
	Moreover, analyzing the proof one can slightly improve the exponent in the term $(1+\ln |e|)^{n+1}$ by writing $(1+\ln |e|)^{n+\varepsilon}$ with $\varepsilon>\min(1/p,1/2).$

	\item In Theorem~\ref{THMmorreynoweight}, the limiting case $\lambda=\frac{n}{p}$ and $q=\infty$ yields the trivial inequality $\| \widehat{f}\|_{L_\infty}\lesssim \|f\|_{L_1}$, since
	\[
	D_{p,\infty}^{\frac{n}{p}}(f)= \sup_{m\in \Z}\sup_{e\subset \Z^n} \frac{(1+\ln |e|)^{n+1}}{|e|^\frac{\beta}{n}}\int_{\cup_{k\in e} Q_k^m} |f(x)|\, dx \asymp \|f\|_{L_1}.
	\]
	We also note that this remark is not applicable for $\lambda=\frac{n}{p}$ and $q<\infty$, since in this case $M_{p,q}^\frac{n}{p}$ only contains the zero function.

	\item We observe that Theorem~\ref{THMmorreynoweight} does not yield any meaningful bound outside the given range of $\lambda$. More precisely, if $0\leq \lambda\leq \max\{0,\frac{n}{p}-\frac{n}{2}\}$ and $0<q\leq \infty$, the expression \eqref{EQfunctionmorrey} is infinite. Indeed, in this case $\beta=\lambda-\max\{0,\frac{n}{p}-\frac{n}{2}\}\leq 0$ and there exists a sequence $\{m_\nu\}\subset \Z$ such that
	\[
	D_{p,q}^\lambda(f)\geq \sup_{e\subset \Z^n} \frac{(1+\ln |e|)^{n+1}}{\big|\cup_{k\in e} Q_k^{m_\nu} \big|^{\frac{1}{p}-\frac{\lambda}{n}}}\int_{\cup_{k\in e} Q_k^{m_\nu}} |f(x)|\, dx \to \infty,\qquad \nu\to \infty,
	\]
	unless $f\equiv 0$, since the cubes $Q_k^m$ may be split into a union of arbitrarily many smaller cubes (i.e., $|e|$ is arbitrarily large), without changing the value of the averaged integral.
	
	\item \label{REM4equivalentdecription}
	Since for any $m\in \Z$ and $e\subset \Z^n$ one has $\big|\cup_{k\in e} Q_k^{m} \big|=2^{mn}|e|$, the supremum in \eqref{EQfunctionmorrey} may  be equivalently written as
	\begin{equation}
		\label{EQsupremumequiv}
		\sup_{e\subset \Z^n}\frac{(1+\ln |e|)^{n+1}}{|e|^{\frac{1}{p}-\max\{0,\frac{1}{p}-\frac{1}{2}\}}}\cdot \frac{1}{(2^{mn})^{\frac{1}{p}-\frac{\lambda}{n}}}\int_{\cup_{k\in e} Q_k^{m}}|f(x)|\, dx.
	\end{equation}

	\item Homogeneity arguments show that if the inequality
	\[
	\| \widehat{f}\|_{M_{p,q}^\lambda} \lesssim  \bigg(\sum_{m\in \Z}\bigg( \sup_{e\subset \Z^n} \frac{ (1+\ln |e|)^{n+1}}{|e|^\frac{\beta}{n}}\cdot \frac{1}{\big| \cup_{k\in e} Q_k^{m}\big|^{\alpha}}\int_{\cup_{k\in e} Q_k^{m}}|f(x)|\, dx\bigg)^q\bigg)^\frac{1}{q}
	\]
	holds for some $\alpha\geq 0$ and $\beta>0$, then necessarily $\alpha=\frac{1}{p}-\frac{\lambda}{n}$.
\end{enumerate}
\end{remark}

	\begin{proof}[Proof of Theorem~\ref{THMmorreynoweight}]
		Let $p\geq 2$. By Hardy-Littlewood inequality  for the Fourier coefficients (cf. \cite{SW}), we have, for each $m\in \Z$ and $k\in \Z^n$,
		\begin{equation}
			\label{EQcoeffinequality}
		\| \widehat{f}\|_{L_p(Q^m_k)} =\bigg( \int_{Q^m_k} |\widehat{f}(x)|^p \, dx\bigg)^\frac{1}{p} = 2^\frac{mn}{p}\bigg( \int_{Q^0_0}\big|\widehat{f}\big(2^mx+2^{m}k\big)\big|^p\, dx\bigg)^\frac{1}{p} \lesssim 2^\frac{mn}{p} \| a(f)\|_{\ell_{p',p}},
		\end{equation}
		where $a(f) = \{a_r\}_{r\in \Z^n}$, and for each $r\in \Z^n$,
		\[
		a_r = \int_{Q^0_0} \widehat{f}\big(2^my+2^{m}k\big) e^{-2\pi i (y,r)}\, dy =  \int_{Q^0_0} \int_{\R^n} f(x) e^{-2\pi i(x,(2^m y+2^{m}k))}\,  e^{-2\pi i (y,r)}\,dx \, dy.
		\]
		We can estimate the Fourier coefficients $a_r$ as follows:
		\begin{align*}
		|a_{r}| &= \bigg| \int_{\R^n} f(x) e^{-2\pi i (x,2^{m} k)} \int_{Q_0^0} e^{-2\pi i (y,2^{m}x+r)}\, dy\, dx \bigg| \\
		&\leq \int_{\R^n} |f(x)| \prod_{j=1}^n  \bigg| \int_{0}^1 e^{-2\pi i y_j (2^m x_j+r_j)}\,dy_j\bigg|\, dx,
		\end{align*}
		where the change of order of integration will  be justified by  $ D_s<\infty$, see \eqref{EQfunctionmorrey}. Since
		\[
		\bigg|\int_0^1 e^{-2\pi i y_j (2^m x_j+r_j)}\,dy_j \bigg|  =\bigg|\frac{1-e^{-2\pi i (2^m x_j+r_j)}}{2\pi (2^m x_j+r_j)}\bigg|,
		\]
		we have
			\begin{equation}
			\label{fff}
		|a_{r}| \leq \frac{1}{(2\pi)^n}\sum_{k\in \Z^n} \int_{Q_k^{-m}}|f(x)| \prod_{j=1}^n  \bigg|\frac{1-e^{-2\pi i 2^m x_j}}{ 2^m x_j+r_j}\bigg|\, dx,
			\end{equation}
		where we have used the decomposition of $\R^n$ into disjoint dyadic cubes of order $-m$, i.e.,
		\[
		\R^n =\bigcup_{k\in \Z^n}Q_k^{-m}.
		\]
		We now observe that for $x_j\in Q_k^{-m}$, we have $|2^m x_j-k_j|\leq 1 $. This implies that for any $k_j,r_j\in \Z$ such that $|k_j+r_j|\geq 2$ we have
		\begin{align*}
		\bigg|\frac{1-e^{-2\pi i 2^m x_j}}{2^m x_j+r_j}\bigg|
\leq \frac{2}{|k_j+r_j|-|2^m x_j-k_j|}\leq \frac{2}{|k_j+r_j|-1}
\leq \frac{4}{ |k_j+r_j|}.
		\end{align*}
		On the other hand, if $|k_j+r_j|< 2$, we have
		\[
		\bigg|\frac{1-e^{-2\pi i 2^m x_j}}{2^m x_j+r_j}\bigg| = \bigg|\frac{1-e^{-2\pi i (2^m x_j+r_j)}}{2^m x_j+r_j}\bigg|\leq 2\pi .
		\]
		Combining these estimates together with \eqref{fff}, we obtain
		\begin{equation}
			\label{EQa_restimate}
		|a_{-r}|\lesssim  (b*c)_r,
		\end{equation}
		where
		\begin{equation}
		\label{EQb}
		b =\bigg\{b_k= \int_{Q_k^{-m}} |f(x)|\, dx\bigg\}_{k\in \Z^n}, \qquad c = \bigg\{c_k= \frac{1}{\prod_{j=1}^n \overline{k}_j}\bigg\}_{k\in \Z^n}.
		\end{equation}
		Combining inequalities \eqref{EQcoeffinequality} and \eqref{EQa_restimate}, and then using Lemmas~\ref{LEMdualnormlorentz}~and~\ref{LEMsuminverses} we obtain the estimate
		\begin{align*}
		\| \widehat{f}\|_{L_p (Q_k^m)}& \lesssim 2^\frac{mn}{p} \| b*c\|_{\ell_{p',p}} \lesssim  2^\frac{mn}{p} \sup_{\| h\|_{\ell_{p,p'}}=1}\sum_{\nu=1}^\infty  (1+\ln \nu)^{n} b_\nu^{**}h_{\nu}^{**}.
		\end{align*}
	Further, applying H\"older's inequality,
	\begin{align*}
	\| \widehat{f}\|_{L_p (Q_k^m)}&\lesssim 2^\frac{mn}{p} \sup_{\|h\|_{\ell_{p,p'}}=1} \sum_{\nu=1}^\infty  (1+\ln \nu)^{n} b_\nu^{**}h_{\nu}^{**} \\
	& \lesssim 	2^\frac{mn}{p}\sup_{\| h\|_{\ell_{p,p'}}=1} \sup_{\nu \in \N}(1+\ln \nu)^{n+1} \nu^\frac{1}{p'}b_\nu^{**}\bigg(\sum_{j=1}^\infty\Big(j^\frac{1}{p}h_j^{**}\Big)^{p'}\frac{1}{j}\bigg)^\frac{1}{p'} \bigg(\sum_{j=1}^\infty \frac{1}{j(1+\ln j)^p} \bigg)^\frac{1}{p}\\
	&\lesssim 2^\frac{mn}{p}\sup_{\nu \in \N}(1+\ln \nu)^{n+1}\nu^\frac{1}{p'}b_\nu^{**},
	\end{align*}
	where we have used that
	\[
	\| h\|_{\ell_{p,p'}} \asymp \bigg(\sum_{j=1}^\infty\Big(j^\frac{1}{p}h_j^{**}\Big)^{p'}\frac{1}{j}\bigg)^\frac{1}{p'},\qquad 1<p< \infty,
	\]
	(cf. \cite[Ch. 4, Lemma 4.5]{BSbook}).	Note that if $e\subset \Z^n$, then we have
	\[
	2^{m(\frac{n}{p}-\lambda)} =\frac{|e|^{\frac{1}{p}-\frac{\lambda}{n}}}{\big| \cup_{k\in e} Q_k^{-m}\big|^{\frac{1}{p}-\frac{\lambda}{n}}},
	\]
	and also
	\begin{equation}
		\label{EQb**}
	b_\nu^{**} = \frac{1}{\nu}\sum_{j=1}^\nu b_j^*  = \frac{1}{\nu} \sup_{\substack{e\subset \Z^n\\|e|=\nu}} \int_{\cup_{k\in e} Q_k^{-m}}|f(x)|\, dx.
	\end{equation}
	We finally have by Remark~\ref{REMcubes} and the above estimates and equalities, for $q=\infty$,
	\begin{align*}
	\|\widehat{f}\|_{M^\lambda_{p,\infty}}& \asymp \sup_{m\in \Z} 2^{-m\lambda} \sup_{k\in \Z^n} \|\widehat{f}\|_{L_p(Q_k^m)} \lesssim \sup_{m\in \Z} 2^{m(\frac{n}{p}-\lambda)}\sup_{\nu\in \N}(1+\ln \nu)^{n+1}\nu^\frac{1}{p'}b_\nu^{**}\\
	&=\sup_{m\in \Z}\sup_{e\subset \Z^n} \frac{ (1+\ln |e|)^{n+1}}{|e|^\frac{\lambda}{n}} \frac{1}{\big| \cup_{k\in e} Q_k^{-m}\big|^{\frac{1}{p}-\frac{\lambda}{n}}}\int_{\cup_{k\in e} Q_k^{-m}}|f(x)|\, dx=D_\infty(f).
	\end{align*}
	For $q<\infty$, we derive
	\begin{align*}
	\|\widehat{f}\|_{M^\lambda_{p,q}}& \asymp \bigg(\sum_{m\in \Z}\bigg( 2^{-m\lambda} \sup_{k\in \Z^n} \|\widehat{f}\|_{L_p(Q_k^m)}\bigg)^q\bigg)^\frac{1}{q}\lesssim \bigg(\sum_{k\in \Z}\bigg( 2^{m(\frac{n}{p}-\lambda)}\sup_{\nu\in \N}(1+\ln \nu)^{n+1}\nu^\frac{1}{p'}b_\nu^{**}\bigg)^q\bigg)^\frac{1}{q}\\
	&=\bigg(\sum_{m\in \Z}\bigg( \sup_{e\subset \Z^n} \frac{ (1+\ln |e|)^{n+1}}{|e|^\frac{\lambda}{n}} \frac{1}{\big| \cup_{k\in e} Q_k^{-m}\big|^{\frac{1}{p}-\frac{\lambda}{n}}}\int_{\cup_{k\in e} Q_k^{-m}}|f(x)|\, dx\bigg)^q\bigg)^\frac{1}{q}=D_{p,q}^\lambda(f).
	\end{align*}
	
	Let now $0< p<2$. From the embedding $M^{\beta}_{2,q}\hookrightarrow M^\lambda_{p,q}$ (Lemma~\ref{LEMembeddingsmorrey}) and the above inequalities we obtain the desired conclusion.
	\end{proof}

\begin{proof}[Proof of Corollary~\ref{Cormorreynoweight}]
First we show that for  $0<  p <\infty$, $0<  q\le \infty$, $0<\lambda\leq \frac{n}{p}$, and $\frac{1}{s}=\frac{1}{p}-\frac{\lambda}{n}$,
	\[D_{p,q}^\lambda(f)=\bigg(\sum_{m\in \Z}\bigg( \sup_{e\subset \Z^n} \frac{ (1+\ln |e|)^{n+1}}{|e|^\frac{\beta}{n}}\cdot \frac{1}{\big| \cup_{k\in e} Q_k^{m}\big|^{\frac{1}{p}-\frac{\lambda}{n}}}\int_{\cup_{k\in e} Q_k^{m}}|f(x)|\, dx\bigg)^q\bigg)^\frac{1}{q}\lesssim \| f\|_{L_{s',q}}.\]
For fixed $m\in \Z$ and $e\subset \Z^n$, since $2^{mn}|e|=\big| \cup_{k\in e} Q_k^{m}\big|$,
\begin{align*}
&\phantom{=}
\frac{ (1+\ln |e|)^{n+1}}{|e|^\frac{\beta}{n}}\cdot \frac{1}{\big| \cup_{k\in e} Q_k^{m}\big|^{\frac{1}{p}-\frac{\lambda}{n}}}\int_{\cup_{k\in e} Q_k^{m}}|f(x)|\, dx \\
&=
 \frac{ (1+\ln |e|)^{n+1}}{|e|^\frac{\beta}{n}}\cdot \frac{1}{2^{mn}|e|}
 (2^{mn}|e|)^{\frac{1}{s'}}
 \int_{\cup_{k\in e} Q_k^{m}}|f(x)|\, dx
\leq
 \frac{ (1+\ln |e|)^{n+1}}{|e|^\frac{\beta}{n}} \Big((2^{mn}|e|)^\frac{1}{s'}f^{**}(2^{mn}|e|) \Big).
\end{align*}
Further, since $f^{**}$ is decreasing,
\begin{align*}
	&\phantom{=}\bigg(\sup_{e\subset \Z^n} \frac{ (1+\ln |e|)^{n+1}}{|e|^\frac{\beta}{n}}\cdot(2^{mn}|e|)^\frac{1}{s'}f^{**}(2^{mn}|e|)\bigg)^q\\
	&\leq \bigg( \sup_{t\geq 0} \sup_{2^t\leq \ell<2^{t+1}} \frac{ (1+\ln \ell)^{n+1}}{\ell^\frac{\beta}{n}}\cdot(2^{mn}\ell)^\frac{1}{s'}f^{**}(2^{mn}\ell)\bigg)^q\lesssim \bigg( \sup_{t\geq 0} \frac{ (1+t)^{n+1}}{2^{t\frac{\beta}{n}}}\cdot(2^{mn+t})^\frac{1}{s'}f^{**}(2^{mn+t})\bigg)^q\\
	&\lesssim \sum_{t=0}^\infty \bigg(\frac{(1+t)^{n+1}}{2^{t \frac{\beta}{n}}}\bigg)^q  \int_{2^{ (m-1)n+t}}^{2^{mn+t}} \Big(x^{\frac{1}{s'}}f^{**}(x)\Big)^q \,\frac{dx}{x}.
\end{align*}
Combining all the  estimates, we get
\begin{align*}
	D_{p,q}^\lambda(f) \lesssim \bigg(\sum_{m\in \Z} \sum_{t=0}^\infty \bigg(\frac{(1+t)^{n+1}}{2^{t \frac{\beta}{n}}}\bigg)^q  \int_{2^{ (m-1)n+t}}^{2^{mn+t}} \Big(x^{\frac{1}{s'}}f^{**}(x)\Big)^q \,\frac{dx}{x}\bigg)^\frac{1}{q} \asymp \|f\|_{L_{s',q}},
\end{align*}
since $\beta>0$. Let us now show that  inequality \eqref{EQmorrey-weaklqineq} is weaker than \eqref{EQthmmorreybasic}, in other words, there exists $f$ for which $D_{p,q}^\lambda(f)<\infty$
   but  $\| f\|_{L_{s',q}}=\infty$. For simplicity, let $n=1.$ For $\frac{1}{s}<1-\alpha<\frac{1}{s}+\beta$, we define
	\[
	f(x)=\begin{cases}
		k^{-\alpha},& \text{if }x\in [2^{k-1},2^{k-1}+1), \, k\in \N,\\
		0,&\text{otherwise}.
	\end{cases}
	\]
	For $m\geq 1$ and $e\subset \Z$, we have
	\begin{align*}
	\frac{1}{\big| \cup_{k\in e} Q_k^{m}\big|^{\frac{1}{p}-\lambda}}\int_{\cup_{k\in e} Q_k^{m}}|f(x)|\,dx &\leq
 \frac{1}{2^{\frac{m}{s}}}
 \int_{1}^{2^{m}}|f(x)|\,dx+
  \frac{1}{(|e|2^{m})^\frac{1}{s}}\sum_{k=2}^{|e|+1} \int_{2^{k}}^{2^{k}+1}|f(x)|\,dx \\
	&\lesssim
 \frac{m^{1-\alpha}}{2^{\frac{m}{s}}}+  \frac{|e|^{1-\alpha}}{(|e|2^{m})^\frac{1}{s}}
.
	\end{align*}
 Therefore,
	\begin{align*}
	&\phantom{=}\sup_{e\subset \Z^n} \frac{ (1+\ln |e|)^{2}}{|e|^\beta}\cdot \frac{1}{\big| \cup_{k\in e} Q_k^{m}\big|^{\frac{1}{p}-\lambda}}\int_{\cup_{k\in e} Q_k^{m}}|f(x)|\,dx \\&\lesssim
 \frac{1}{2^\frac{m}{s}}
\sup_{e\subset \Z^n}
\Big(
 {m^{1-\alpha}}+
\frac{ (1+\ln |e|)^{2}}{|e|^{\beta+\frac{1}{s}+\alpha-1}}
 \Big)
 \lesssim
 \frac{{m^{1-\alpha}}}{2^\frac{m}{s}}.
\end{align*}
	For $m\leq 0$ and $e\subset \Z$, let $\nu\in \N$ be such that $\nu 2^{-m}\leq |e|<(\nu+1)2^{-m}$. Then,
	\[
	\int_{\cup_{k\in e} Q_k^{m}}|f(x)|\, dx\leq \sum_{\ell=1}^\nu \int_{2^\ell}^{2^{\ell}+1}|f(x)|\, dx=\sum_{\ell=1}^\nu t^{-\alpha} \asymp \nu^{1-\alpha} \leq (2^m |e|)^{1-\alpha},
	\]
	and so,
	\begin{align*}
	&\phantom{=}\sup_{e\subset \Z^n} \frac{ (1+\ln |e|)^{2}}{|e|^\beta}\cdot \frac{1}{\big| \cup_{k\in e} Q_k^{m}\big|^{\frac{1}{p}-\lambda}}\int_{\cup_{k\in e} Q_k^{m}}|f(x)|\,dx \\
	&\leq \phantom{=}\sup_{e\subset \Z^n} \frac{ (1+\ln |e|)^{2}}{|e|^\beta}\cdot \frac{(2^m |e|)^{1-\alpha}}{(2^m|e|)^{\frac{1}{s}}}\asymp 2^{ m(1-\alpha-\frac{1}{s})}.
	\end{align*}
	Finally, we derive that
	\[
	D_{p,q}^\lambda(f)\lesssim \bigg(\sum_{m=-\infty}^0 2^{mq(1-\alpha-\frac{1}{s})}  + \sum_{m=1}^\infty \bigg(\frac{m^{1-\alpha}}{2^\frac{m}{s}}\bigg)^q \bigg)^{\frac{1}{q}}<\infty,
	\]
	for $0<q<\infty$ (note that also $D_{p,\infty}^\lambda(f)<\infty$).	On the other hand, it is clear that  $\|f\|_{L_{s',\infty}}=\infty$, since $f^*(t)\asymp t^{-\alpha}$ as $t\to\infty$.
\end{proof}

\begin{remark}
We note that the inequality
\[
\| \widehat{f}\|_{M_{p,q}^\lambda}\lesssim \| {f}\|_{L_{s',q}}, \qquad 0<q\le \infty, \qquad \frac{1}{s}=\frac{1}{p}-\frac{\lambda}{n},
\]
can be also obtained interpolating  both sides of
$\| \widehat{f}\|_{M_{p,\infty}^\lambda}  \lesssim\| f\|_{L_{s',\infty}}$
and using that
 		\begin{equation*}\label{12}
 			\big( M_{p,\infty}^{\lambda_0},  M_{p,\infty}^{\lambda_1} \big)_{\theta,q} \hookrightarrow   M_{p,q}^{\lambda},\qquad
 \lambda =(1-\theta)\lambda_0 + \theta \lambda_1,\quad \theta \in (0,1),
 		\end{equation*}
	provided $0<p <\infty$, $\theta \in (0,1)$, and $0<q\leq\infty$ (see  Lemma \ref{T1}).
\end{remark}

\subsection{The case of monotone functions}
We observe that under certain assumptions on $f$, the Lorentz norms $\| f\|_{L_{s',q}}$ reduce to  weighted $L_q$ norms. A function $g:(0,\infty)\to\C$ that is locally of bounded variation is said to be general monotone (written $g\in GM$, see \cite{nachr}) if there exist constants $C,\lambda>0$ such that
\[
\int_{x}^{2x} |dg(t)|\leq \frac{C}{x}\int_{x/\lambda}^{\lambda x}|g(t)|\, dt.
\]
\begin{corollary}\label{CORgm}
	Let a radial function $f(x)=f_0(|x|)$ be such that $f_0\in GM$. For $0<p<\infty$, $1\leq q\leq \infty$, and $\max\{0,\frac{n}{p}-\frac{n}{2}\}<\lambda<\frac{n}{p}$,  we have
	\begin{equation}
		\label{EQmorrey-infty}
	\| \widehat{f}\|_{M^\lambda_{p,q}}\lesssim \| |x|^{\frac{n}{p'}-\frac{n}{q}+\lambda} f\|_{L_q(\R^n)}.
	\end{equation}
\end{corollary}
\begin{proof}
We note that for all $0<p<\infty$ and $\max\{0,\frac{n}{p}-\frac{n}{2}\}<\lambda<\frac{n}{p}$, there holds
\[
0<\frac{1}{s}:= \frac{1}{p}-\frac{\lambda}{n}<\frac{1}{2}.
\]
By Corollary~\ref{Cormorreynoweight}, $\| \widehat{f}\|_{M^\lambda_{p,q}}\lesssim  \| f\|_{L_{s',q}}$. Taking into account that $f^*(t\omega_n)=f_0^*\big(t^{\frac{1}{n}} \big)$ with
$\omega_n$ being the volume of the unit ball in $\R^n$,
 it is straightforward to see that
\[
\| f\|_{L_{s',q}} \asymp \| f_0\|_{L_{\frac{s'}{n},q}}\asymp \| t^{\frac{n}{s'}-\frac{1}{q}} f_0\|_{L_q(\R_+)} \asymp \| |x|^{\frac{n}{p'}-\frac{n}{q}+\lambda} f\|_{L_q(\R^n)},
\]
where the second to last equivalence (valid for $q\geq 1$, $s>1$, and $f_0\in GM$) was obtained in \cite[Theorem 3]{bootonlorentz}.
\end{proof}

\begin{remark}
	Let $q=\infty$. We formally inspect the important limiting cases for $\lambda$ in Corollary~\ref{CORgm} corresponding to Lebesgue spaces (i.e., $\lambda=\frac{n}{p}$ or $\lambda=0$ with $p\geq 2$).
\begin{enumerate}[wide = 0pt]
	\item  For $\lambda=n/p$, \eqref{EQmorrey-infty} is
	\[
	\| \widehat{f}\|_{L_\infty}	\lesssim \| |x|^{n} f\|_{L_\infty}.
	\]
	This inequality is clearly not true, as shown by the example $f_0(t)=\chi_{(0,1)}(t)+t^{-n}\chi_{(1,N)}(t)$.
	\item For $p\geq 2$ and $\lambda=0$, \eqref{EQmorrey-infty} is
	\begin{equation}
	\label{EQmorrey-infty2}
	\| \widehat{f}\|_{L_p}\lesssim \sup_{t>0} \big(t^{\frac{n}{p'}}|f_0(t)| \big),
	\end{equation}
	which does not hold for any $n\in \N$. To see this, we recall the Hardy-Littlewood inequality \cite{DBoasHankel, GLT} (obtained under our assumptions), which reads as
	\begin{equation}
	\label{EQHLineq}
	\|\widehat{f}\|_{L_p}\asymp \bigg(\int_0^\infty t^{n(p-2)}|f_0(t)|^p\, dt\bigg)^{\frac{1}{p}},\qquad \frac{2n}{n+1}<p<\infty,
	\end{equation}
	Thus, the function $f$ with radial part $f_0(t)=t^{-\frac{n}{p'}}\chi_{(0,N)}(t)$ provides a  counterexample for \eqref{EQmorrey-infty2} (note that since $p\geq 2$, for every $n\in \N$, $p$ lies in the range given in \eqref{EQHLineq}).
\end{enumerate}
\end{remark}

\subsection{Inequalities in weighted Morrey spaces}

We now present a weighted version of Theorem~\ref{THMmorreynoweight}.

\begin{theorem}\label{THMmorreyweighted}
	Let $0< p<\infty$, $0<q\leq \infty$, and $u\in \Xi_{p,q}^n$ be such that $u(r)\asymp u(2r)$ for $r>0$.
Additionally, if $0< p<2$, we assume $r^{\frac{n}{p}-\frac{n}{2}}u(r)\in \Xi_{2,q}^n$.
Suppose that
	\[
	D_{p,q}^u(f):= \bigg( \sum_{m\in \Z} \bigg(\sup_{e\subset \Z^n} \frac{(1+\ln|e|)^{n+1}}{|e|^{\frac{1}{p}-\max\{0,\frac{1}{p}-\frac{1}{2}\}}}\cdot \frac{u\big(2^{-m}\big)}{(2^{mn})^\frac{1}{p}}\int_{\cup_{k\in e} Q_k^m} |f(x)|\, dx\bigg)^q\bigg)^\frac{1}{q} <\infty
	\]
	\textnormal{(}with the usual modification if $q=\infty$\textnormal{)}. Then $\widehat{f}\in M^u_{p,q}$, and moreover
	\[
	\| \widehat{f}\|_{M^u_{p,q}}\lesssim D_{p,q}^u (f).
	\]
\end{theorem}
By \eqref{EQsupremumequiv}, it is clear that Theorem~\ref{THMmorreyweighted} with $u(r)=r^{-\lambda}$, $\max\{0,\frac{n}{p}-\frac{n}{2}\}<\lambda \leq\frac{n}{p}$, reduces to Theorem~\ref{THMmorreynoweight}.
\begin{proof}[Proof of Theorem~\ref{THMmorreyweighted}]
	We prove in detail only the case $q<\infty$; the case $q=\infty$ is treated similarly. Let first $p\geq 2$. As in the proof of Theorem~\ref{THMmorreynoweight}, we have $\displaystyle\| \widehat{f}\|_{L_p (Q_k^m)} \lesssim  2^\frac{mn}{p}\sup_{\nu \in \N}(1+\ln \nu)^{n+1}\nu^\frac{1}{p'}b_\nu^{**}$, 	 where $b_k$ is defined as in \eqref{EQb}. Further, since $u(r)\asymp u(2r)$ for $r>0$, we have
	\[
	\| f\|_{M^u_{p,q}}\lesssim \bigg(\sum_{m\in\Z} \Big(u(2^m)\sup_{k\in \Z}\| f\|_{L_p(Q_k^m)} \Big)^q\bigg)^\frac{1}{q}  \lesssim \bigg(\sum_{m\in\Z} \Big(u(2^m) 2^\frac{mn}{p}\sup_{\nu \in \N}(1+\ln \nu)^{n+1}\nu^\frac{1}{p'}b_\nu^{**}   \Big)^q\bigg)^\frac{1}{q}
	\]
	(see Remark~\ref{REMcubes} for the first inequality). Finally, by  \eqref{EQb**}, we conclude that
	\begin{align*}
	\|\widehat{f}\|_{M^u_{p,q}}&\lesssim \bigg(\sum_{m\in\Z} \bigg(u(2^m) 2^\frac{mn}{p}\sup_{\nu \in \N}(1+\ln \nu)^{n+1}\nu^\frac{1}{p'}\cdot \frac{1}{\nu} \sup_{\substack{e\subset \Z^n\\|e|=\nu}} \int_{\cup_{k\in e} Q_k^{-m}}|f(x)|\, dx   \bigg)^q\bigg)^\frac{1}{q}\\
	&= \bigg( \sum_{m\in \Z} \bigg(\sup_{e\subset \Z^n} \frac{(1+\ln|e|)^{n+1}}{|e|^{\frac{1}{p}}}\cdot \frac{u\big(2^{-m}\big)}{(2^{mn})^\frac{1}{p}}\int_{\cup_{k\in e} Q_k^m} |f(x)|\, dx\bigg)^q\bigg)^\frac{1}{q}  = D_{p,q}^u(f).
	\end{align*}
	Let now $0<p<2$. As in the first part of Lemma~\ref{LEMembeddingsmorrey}, a simple application of H\"older's inequality shows that $M_{2,q}^v\hookrightarrow M_{p,q}^u$, with
	$
	v(r)= r^{ \frac{n}{p}-\frac{n}{2}}u(r),
	$
	so that $\| \widehat{f}\|_{M_{p,q}^u}\lesssim \|\widehat{f}\|_{M_{2,q}^v}\lesssim D_{2,q}^v(f)=D_{p,q}^u(f)$, as desired.
\end{proof}

Analogously as in Corollary~\ref{Cormorreynoweight}, we can estimate $D_{p,q}^u(f)$ in Theorem~\ref{THMmorreyweighted} by a weighted Lorentz norm. Given $0<q\leq\infty$ and a measurable function $v:(0,\infty)\to (0,\infty)$, we define the weighted Lorentz space $\Gamma_q^v(\R^n)=\Gamma_q^v$ by
\[
\| f\|_{\Gamma_q^v}:= \| v f^{**}\|_{L_q(0,\infty)},
\]
see \cite{Sawyer}.
\begin{corollary}\label{CORweightedlorentz}
	Let $0< p<\infty$, $0<q\leq \infty$, and let $u:(0,\infty)\to(0,\infty)$ be such that $u(r)\asymp u(2r)$ for $r>0$. Assume, moreover, that there exists $\alpha>\max\{0,\frac{n}{p}-\frac{n}{2}\}$ such that $x^\alpha u(x)$ is almost decreasing, that is,
	\begin{equation}
		\label{EQlorentzweightcond}
		x^\alpha u(x)\lesssim y^\alpha u(y),\qquad \text{for all }y<x.
	\end{equation}
	Then,
	\[
	\| \widehat{f}\|_{M^u_{p,q}}  \lesssim \| f\|_{\Gamma_q^v},
	\]
	where $v(x)=u\big( x^{-\frac{1}{n}}\big)x^{\frac{1}{p'}-\frac{1}{q}}$.
\end{corollary}
\begin{proof}
	The proof is similar to that of Corollary~\ref{Cormorreynoweight}. Thus, we only sketch the steps where it is slightly different.
It follows from \eqref{EQlorentzweightcond} that
$ u(x)\lesssim \kappa^\alpha u(\kappa x)$ for any $0<\kappa<1$.
Then for fixed $m\in \Z$ and $e\subset \Z^n$,
	\[
	u(2^{-m})\lesssim |e|^{-\frac{\alpha}{n}} u\big(2^{ -m}|e|^{-\frac{1}{n}}\big)= |e|^{-\frac{\alpha}{n}} u\big((2^{mn}|e|)^{-\frac{1}{n}} \big).
	\]
	Taking into account that $2^{mn}|e|=\big| \cup_{k\in e} Q_k^{m}\big|$, we get, by setting $\rho=\max\{0,\frac{n}{p}-\frac{n}{2}\}$,
	\begin{align*}
		&\phantom{=}\sup_{e\subset \Z^n}\frac{(1+\ln|e|)^{n+1}}{|e|^{\frac{1}{p}-\max\{0,\frac{1}{p}-\frac{1}{2}\}}}\cdot \frac{u\big(2^{-m}\big)}{(2^{mn})^\frac{1}{p}}\int_{\cup_{k\in e} Q_k^m} |f(x)|\, dx\\
		&\lesssim  \sup_{e\subset \Z^n}\frac{(1+\ln |e|)^{n+1}}{|e|^{\frac{\alpha-\rho}{n}}} u\big((2^{mn}|e|)^{-\frac{1}{n}} \big) (2^{mn}|e|)^{\frac{1}{p'}}f^{**}(2^{mn}|e|).
	\end{align*}
	Thus, for $q=\infty$, the result immediately follows from the last estimate and Theorem~\ref{THMmorreyweighted} (see also \eqref{EQsupremumequiv}). On the other hand, for $q<\infty$, we derive that
	\begin{align*}
		D_{p,q}^u(f)&\lesssim  \bigg( \sum_{m\in \Z}\bigg( \sup_{t\geq 0} \sup_{2^t\leq \ell<2^{t+1}} \frac{(1+\ln \ell)^{n+1}}{ \ell^\frac{\alpha-\rho}{n}} u\big((2^{mn}\ell)^{-\frac{1}{n}} \big) (2^{mn} \ell)^{\frac{1}{p'}}f^{**}(2^{mn} \ell) \bigg)^q\bigg)^\frac{1}{q}\\
		&\lesssim  \bigg( \sum_{m\in \Z} \sum_{t=0}^\infty \bigg(\frac{(1+t)^{n+1}}{  2^{t\frac{ \alpha-\rho}{n}}}  \bigg)^q \int_{2^{(m-1)n+t}}^{2^{mn+t}} \Big(u\big(x^{-\frac{1}{n}}\big) x^{\frac{1}{p'}}f^{**}(x) \Big)^q\frac{dx}{x}\bigg)^\frac{1}{q} \asymp \| f\|_{\Gamma_q^v},
	\end{align*}
	since $u(r)\asymp u(2r)$ for $r>0$. The conclusion follows from Theorem~\ref{THMmorreyweighted}.
\end{proof}

It is clear that for $u(r)=r^{-\lambda}$ and $1\leq q\leq \infty$ Corollary~\ref{CORweightedlorentz} corresponds to Corollary~\ref{Cormorreynoweight}. Indeed, in this case $M_{p,q}^u=M_{p,q}^\lambda$ and
\[
\| f\|_{\Gamma_q^v} = \bigg( \int_0^\infty \Big(x^{\frac{\lambda}{n}+\frac{1}{p'}} f^{**}(x)\Big)^q\frac{dx}{x} \bigg)^\frac{1}{q} =\| f\|_{L_{s',q}}, \quad
  \quad \frac{1}{s}= \frac{1}{p}-\frac{\lambda}{n},
\]
where the condition $2<s<\infty$ follows from \eqref{EQlorentzweightcond}.

\vspace{0.4mm}
\section{Pitt's inequalities between Morrey and Lebesgue spaces}\label{section5}
The classical Pitt inequality states that, for $1<p\leq \xi<\infty$,
	\begin{equation}\label{pitt-cl}
	\| |x|^{-\delta} \widehat{f}\|_{L_{\xi}}\lesssim \| |x|^\gamma f\|_{L_p}
		\end{equation}
	holds if and only if
	\begin{equation}
		\label{EQclasspittconds}
		\max\bigg\{ 0, n\bigg( \frac{1}{\xi}-\frac{1}{p'} \bigg)  \bigg\}\leq \delta<\frac{n}{\xi}, \qquad
		\gamma =\delta+n\bigg(\frac{1}{p'}-\frac{1}{\xi}\bigg),
	\end{equation}
	see, e.g., \cite{BH,GLTPitt}. Its analogue for weighted Morrey norms does not hold in general, see \eqref{EQweightedmorreymorrey} in Appendix.

 In this section we obtain {\it weighted} Fourier inequalities between   Morrey and Lebesgue spaces, generalizing the classical Pitt inequality. Our proof is based on the combination of (\ref{pitt-cl}) and  the embedding results for Morrey spaces, see e.g. \cite[Theorem 3.12]{HS}.

\begin{theorem}\label{CORpittmorrey-fullrange}
	Let $1<p<\infty$, $0<q<\infty$, and $0\leq\nu<\frac{n}{q}$ be such that
	\[
	\frac{nq}{n-\nu q}\geq p,\qquad  \text{or equivalently}, \qquad \frac{1}{p}\geq \frac{1}{q}-\frac{\nu}{n}.
	\]
	Then the inequality
	\begin{equation}
		\label{EQpittembedding}
		\| |x|^{-\delta} \widehat{f}\|_{M_{q}^{\nu}}\lesssim \| |x|^\gamma f\|_{L_p}
	\end{equation}
	holds if and only if

	\begin{equation}
		\label{EQparamscondpitt}
	\max\bigg\{ 0, n\bigg(\frac{1}{q} -\frac{1}{p'} \bigg)-\nu \bigg\}\leq \delta<\frac{n}{q}-\nu, \qquad
	\gamma =\delta+n\bigg(\frac{1}{p'}-\frac{1}{q}\bigg)+\nu.
		\end{equation}
\end{theorem}
It is clear that for $\nu=0$ we recover the  classical Pitt inequality (\ref{pitt-cl}) and
for $\delta=\gamma=0$  estimate (\ref{EQMorrey-Lorentzembedding}).

\begin{proof}
We recall that, for   $0\leq \nu<\frac{n}{q}$ and $0<q,\xi<\infty$, the inequality
	\[
	\| g\|_{M_{q}^{\nu}}\lesssim \| g\|_{L_\xi}
	\]
	holds if and only if
	\[
	\xi= \frac{nq}{n-\nu q}, \qquad \text{or equivalently}, \qquad \frac{1}{\xi}= \frac{1}{q}-\frac{\nu}{n}.
	\]
	In particular, we note that $q\leq \xi$. Applying this result and using Pitt's inequality (\ref{pitt-cl}), we get
	\[
	\| |x|^{-\delta} \widehat{f}\|_{M_{q}^{\nu}}\lesssim 	\| |x|^{-\delta} \widehat{f}\|_{L_\xi}\lesssim  \| |x|^\gamma f\|_{L_p},
	\]
under the condition
	$
	\frac{1}{\xi}= \frac{1}{q}-\frac{\nu}{n}
	$
and
\eqref{EQclasspittconds}.

	Let us now prove the necessity part. The second condition in \eqref{EQparamscondpitt} easily follows by  homogeneity. Indeed, for an integrable $f$ with $\| |x|^\gamma f\|_{L_p} <\infty$, take $f_\lambda(x)=f(\lambda x)$ with positive $\lambda.$ Then by \eqref{EQparamscondpitt},
	\begin{align*}
	\lambda^{-n-\nu-\delta+\frac{n}{q}} \| |y|^{-\delta}\widehat{f}\|_{L_q(B_1(0))}&\asymp \lambda^{-\nu} \||y|^{-\delta}\widehat{f_\lambda}\|_{L_q(B_\lambda(0))} \leq \||y|^{-\delta}\widehat{f_\lambda}\|_{M_q^\nu}\\
	&\lesssim \| |x|^\gamma f_\lambda\|_{L_p} \asymp \lambda^{-\frac{n}{p}-\gamma}\| |x|^\gamma f\|_{L_p}.
	\end{align*}
and hence, $\gamma=\delta+n(\frac{1}{p'}-\frac{1}{q})+\nu.$

	We now show that \eqref{EQpittembedding} implies $\max\big\{ 0, n\big(\frac{1}{q} -\frac{1}{p'} \big)-\nu \big\}\leq \delta$. Let, for $N>0$ and $t\in \R$, $f_N(x)=e^{2\pi iNx}\chi_{(1,2)}(x)$, and
 $
	\displaystyle F_N(x) = \prod_{j=1}^n f_N(x_j),$ $x\in \R^n.
	$
	Clearly, $\| |x|^\gamma F_N\|_{L_{p}}\asymp 1$ for any $\gamma\in \R$. On the other hand, since
$ \displaystyle
	\widehat{f}_N(s)=\frac{e^{-2\pi i(s-N)}(1-e^{-2\pi i(s-N)})}{2\pi i(s-N)}$ and $ \displaystyle\widehat{F}_N(x) = \prod_{j=1}^n \widehat{f}_N(x_j),
	$
 we get
	\begin{align*}
		\| |x|^{-\delta} \widehat{F}_N \|_{M_q^\nu}&\gtrsim \bigg( \int_{|x_1-N|\leq \frac1{2\pi}}
 \int_{|x_2-N|\leq \frac1{2\pi}}\cdots \int_{|x_N-N|\leq \frac1{2\pi}} |x|^{-\delta q}\prod_{j=1}^n \bigg(\frac{|1-e^{-2\pi i(x_j-N)}|}{2\pi|x_j-N|}\bigg)^q dx  \bigg)^{1/q}\\
		&\asymp N^{-\delta}.
	\end{align*}
	Putting all estimates together, \eqref{EQpittembedding} implies
	$N^{-\delta}\lesssim1,$ that is, $\delta\geq 0$.
	
Further, we observe that the condition  $\delta\geq n\big(\frac{1}{q}-\frac{1}{p'}\big)-\nu$ is equivalent to $\gamma\geq 0$. Consider, for $N>1$ and $t\in \R$, $g_N(t)=\chi_{(N,N+1)}(t)$, and
	$ \displaystyle
	G_N(x)=\prod_{j=1}^n g_N(x_j),$ $ x\in \R^n.
	$
	In this case \eqref{EQpittembedding} implies $\displaystyle 1\lesssim \| |x|^{-\delta} \widehat{G}_N \|_{M_q^\nu}\lesssim \| |x|^\gamma G_N\|_{L_p}\asymp N^\gamma$. Thus, $\gamma\geq 0$.
	
	Finally, we prove that $\delta<\frac{n}{q}-\nu$. Let
	\[
	f(x)=\frac{|x|^{-\gamma-\frac{n}{p}}}{|\log |x||}\chi_{B_{\frac{1}{ 2\pi
}}(0)}(x).
	\]
	It is clear that $\| |x|^\gamma f\|_{L_p}<\infty$ for $p>1$. Since $f$ is a nonnegative function supported on $B_{\frac{1}{2\pi }}(0)$, we have, for $|y|\leq 1$,
	\[
	|\widehat{f}(y)|\geq \bigg|\int_{B_{\frac{1}{2\pi }}(0)} f(x)\cos (2\pi (x, y))\, dx\bigg| \asymp \int_{B_{\frac{1}{2\pi }}(0)} \frac{|x|^{-\gamma-\frac{n}{p}}}{|\log |x||}\, dx \asymp \int_0^{1}
  \frac{s^{n-1-\gamma-\frac{n}{p}}}{|\log s|}\, ds.
	\]
	Thus, it follows from \eqref{EQpittembedding} and the condition $\delta\geq0$ that
	\[
	\infty >\| |x|^\gamma f\|_{L_p}  \gtrsim \| |x|^{-\delta}\widehat{f}\|_{M^\nu_q} \geq \| \widehat{f}\|_{L_q(B_1(0))}\asymp \int_0^{1}  \frac{s^{n-1-\gamma-\frac{n}{p}}}{|\log s|}\, ds.
	\]
Therefore, we have $\gamma<\frac{n}{p'}$, {or equivalently,} $\delta<\frac{n}{q}-\nu.$
\end{proof}

\begin{remark}
We do not study in this paper Morrey--Lorentz analogues of Pitt's inequality \eqref{EQpittembedding}, see \cite{NDT}.
\end{remark}

	\section{Fourier inequalities in Campanato spaces}\label{SECfouriercampanato}

The  main result in this section is the following.

\begin{theorem}\label{THMcampanatoweighted}
	Let $1\leq p<\infty$, $0<q\leq \infty$,  and let  $v,w:(0,\infty)\to (0,\infty)$ be weight functions satisfying $v(r)\asymp v(2r)$, $w(r)\asymp w(2r)$  and such that the estimates
	\begin{equation}
		\label{EQweightedcampanato}
		\sup_{r>0} r^\frac{n}{p} w(r)v(1/r)<\infty,
	\end{equation}
	and
	\begin{equation}
	\label{EQweightedcampanatocond2-}
	\sum_{k=-\infty}^m 2^{k} v(2^k)\lesssim  2^{m}v(2^m),\qquad \sum_{k=m}^\infty  v(2^{k})\lesssim v(2^m),\qquad m\in \Z,
	\end{equation}
	hold. Then, the inequality
\begin{equation}
	\label{loctrMorrey}
	| \widehat{f}|_{C_{p,q}^w} \lesssim  \bigg(\sum_{k\in \Z} \bigg( \frac{1}{v(2^k)} \int_{B_{2^{k+1}}(0)\backslash B_{2^k}(0)}|f(x)|\, dx\bigg)^q\bigg)^\frac{1}{q},
	\end{equation}
	holds \textnormal{(}with the usual modification for $q=\infty$\textnormal{)}.
\end{theorem}

The right hand side of \eqref{loctrMorrey} is a weighted truncated Lebesgue  norm, cf. \eqref{localtruncated}.
Taking $v(r)=r^{-\alpha}$ with $0<\alpha=\lambda-\frac{n}{p}<1$ and $w(r)=r^{-\lambda}$, we obtain the following result for the classical Campanato space.

\begin{corollary}\label{THMcampanato}
	Let $1\le p<\infty$, $0<q\leq \infty$, and $0<\alpha=\lambda-\frac{n}{p}<1$. We have
	\[
	| \widehat{f}|_{C_{p,q}^\lambda}\lesssim
 \|  f\|_{T^{\alpha}_{q}L_1}.
	\]
\end{corollary}

This and  Proposition 		\ref{lem8} (cf. \eqref{ggg}) immediately imply
\begin{remark}\label{THMcampanato'}
	Let $1\le p<\infty$, $1\leq q\leq \infty$, and $0<\alpha=\lambda-\frac{n}{p}<1$. We have
 \[
  \int_0^1 \big( t^{-\alpha} \omega(\widehat{f},t)_
  \infty\big)^q\,  \frac{dt}t
\lesssim\sum_{k\in \Z} \Big( 2^{\alpha k} \|f\|_{L_1( B_{2^{k+1}}(0)\backslash B_{2^k}(0))} \Big)^q.
\]
In particular, we have
 $
  |\widehat{f}|_{B_{\infty,1}^\alpha}
\lesssim   \big\||x|^\alpha f\big\|_1
$ and
 $
  |\widehat{f}|_{\textnormal{Lip}\, \alpha}
  \lesssim\sup\limits_{k\in \Z}
  \int\limits_{{2^{k}}<|x|<{2^{k+1}}}
  |x|^\alpha|f(x)|dx.
$
We note that estimating the Lipschitz norm of the Fourier transform is a classical problem in Fourier analysis, see \cite{bray, tik, sampson,  tit, zygmund}. Our estimate  refines the known one-dimensional result recently obtained \cite{volos}.
\end{remark}

\begin{proof}[Proof of Theorem~\ref{THMcampanatoweighted}]
	We start with the case $q=\infty$, which will be used to prove the general case by means of an interpolation argument.
	
We note that condition \eqref{EQweightedcampanatocond2-} has  the following self-improving property:
for any  weights satisfying \eqref{EQweightedcampanatocond2-},
  there exists a positive $\gamma$ such that
	\begin{equation}
	\label{EQweightedcampanatocond2}
	\sum_{k=-\infty}^m 2^{k(1-\gamma)} v(2^k)\lesssim  2^{m(1-\gamma)}v(2^m),\qquad \sum_{k=m}^\infty 2^{k\gamma} v(2^{k})\lesssim 2^{m\gamma}v(2^m),\qquad m\in \Z.
	\end{equation}

	We observe that it follows from \eqref{EQweightedcampanatocond2} that  $r^{\gamma} v(r)$ is almost decreasing, i.e.,
	\begin{equation}
		\label{EQalmostdecv_0}
		t^{\gamma}v(t)\lesssim s^{\gamma}v(s),\qquad \text{for }s\leq t.
	\end{equation}
	Then, of course, $v(r)$ is almost decreasing as well. Similarly, $r^{1-\gamma}v(r)$ is almost increasing, i.e.,
	\begin{equation}
		\label{EQalmostincv_0}
		s^{1-\gamma}v(s)\lesssim t^{1-\gamma}v(t),\qquad \text{for }s\leq t,
	\end{equation}
	as well as  $rv(r)$.
	
	 For $r>0$, we have
	\begin{align*}
		\int_{B_{r}(0)}|\widehat{f}(x+\xi)-A_r \widehat{f}(\xi)|^{p}\,  d x&  = \int_{B_{r}(0)} \bigg| \frac{1}{|B_{r}|} \int_{B_{r}(0)} \int_{\mathbb{R}^{n}} f(z)\big(e^{-2\pi i(x+\xi, z)}-e^{-2\pi i(y+\xi, z)}\big) d z \, d y\bigg|^{p} d x
		\\
		&\leq \int_{B_{r}(0)}\bigg(\frac{1}{|B_{r}|} \int_{B_{r}(0)} \int_{\mathbb{R}^{n}} 2|f(z)|\big|\sin {(\pi(x-y, z))}\big| d z \,d y\bigg)^{p} d x.
	\end{align*}
		Using that  $|\sin t|\leq \min\{1,|t|\}$, we continue the estimate as follows:
	\begin{align}
		&{\lesssim}
		\int_{B_{r}(0)}\bigg(\frac{1}{|B_{r}|} \int_{B_{r}(0)} \int_{|z|<\frac{1}{|x-y|}}|f(z) \| z||x-y|\, d z\, d y\bigg)^{p} d x\nonumber\\
		&\phantom{=}+ \int_{B_{r}(0)}\bigg(\frac{1}{|B_{r}|} \int_{B_{r}(0)} \int_{|z|>\frac{1}{|x-y|}}|f(z)|\, d z\, d y\bigg)^{p} d x\nonumber
		\\
		&\lesssim \bigg(\int_{B_{r}(0)}\bigg(\frac{1}{|B_{r}|} \int_{B_{r}(0)}v\big(|x-y|^{-1}\big)\, d y\bigg)^{p} d x\bigg)\nonumber\\		
		&\phantom{=} \times
		\bigg(\sup_{r>0}\bigg(\frac{1}{r v(r)}\int_{B_r(0)} |x||f(x)|\, dx + \frac{1}{v(r)} \int_{\R^n \backslash B_r(0)} |f(x)|\, dx\bigg)\bigg)^p.\label{EQsupcampanato}
	\end{align}
	On the one hand, Minkowski's inequality  and \eqref{EQweightedcampanato} yield
	\begin{align}\label{EQsupcampanato--}
		\int_{B_r(0)} \bigg(\frac{1}{\left|B_{r}\right|}\int_{B_r(0)} v\big(|x-y|^{-1}\big)\, dy\bigg)^p\, dx &\leq \int_{B_{2r}(0)}v\big(|x|^{-1}\big)^p\, dx\lesssim r^n v(1/r)^p
		\lesssim  \frac{1}{w(r)^p},
	\end{align}
	where we have used the fact that $v$ is almost decreasing and $v(r)\asymp v(2r)$ for $r>0$. Further, the supremum in \eqref{EQsupcampanato} is equivalent to
	\begin{align*}
		&\phantom{=}\sup_{m\in \Z}\bigg(\frac{1}{2^m v(2^{m})}\sum_{k=-\infty}^m \Big(\frac{ 2^kv(2^k)}{v(2^k)} \int_{B_{2^k}(0)\backslash B_{2^{k-1}}(0)} |f(x)|\, dx\Big)\\
		&\phantom{=} +\frac{1}{v(2^m)}\sum_{k=m}^{\infty} \Big(\frac{v(2^k)}{v(2^k)}\int_{B_{2^{k+1}}(0)\backslash B_{2^k}(0)}|f(x)|\, dx\Big)\bigg)\lesssim \sup_{k\in \Z} \frac{1}{v(2^k)} \int_{B_{2^{k+1}}(0)\backslash B_{2^k}(0)}|f(x)|\, dx,
	\end{align*}
	where the last inequality is a consequence of \eqref{EQweightedcampanatocond2}. Combining the above estimates, the result follows for the case $q=\infty$.
	
	For the case $q<\infty$, by \eqref{EQequivcampanatoinf}, the monotonicity on $r$ of $\inf_{c} \|f-c\|_{L_p(B_{r}(x))}$, and the equivalence $w(r)\asymp w(2r)$, we get
	\[
	|\widehat{f}|_{C_{p,q}^w}\asymp  \bigg(\sum_{k\in \Z} \Big(w(2^k) \sup_{x\in \R^n}  \|\widehat{f}- A_{2^k}{\widehat{f}}(x)\|_{L_p(B_{2^k}(x))}\Big)^q \bigg)^\frac{1}{q}.
	\]
	Writing
	\[
	f_{0,k}(x) = f(x)(1-\chi_{B_{2^{-k}}(0)}(x)), \qquad f_{1,k}(x)= f(x)\chi_{B_{2^{-k}}(0)}(x),
	\]
	for $k\in \Z$, it follows from the last estimate that
	\begin{equation*}
		|\widehat{f}|_{C_{p,q}^w}\lesssim  \bigg( \sum_{k\in \Z} \bigg( \frac{w(2^k)}{w_0(2^k)} |\widehat{f_{0,k}}|_{C_{p,\infty}^{w_0}}\bigg)^q \bigg)^\frac{1}{q} + \bigg( \sum_{k\in \Z} \bigg( \frac{w(2^k)}{w_1(2^k)} |\widehat{f_{1,k}}|_{C_{p,\infty}^{w_1}}\bigg)^q \bigg)^\frac{1}{q},
	\end{equation*}
	with $w_j(t)=t^{-\frac{n}{p}}v_j(1/t)^{-1}$, where $v_0(t)=t^{\gamma}v(t)$ and $v_1(t)=t^{-\gamma}v(t)$.
Note that, by \eqref{EQsupcampanato} and  \eqref{EQsupcampanato--}, we have
\begin{align}\label{EQsupcampanato--+}
|\widehat{f}|_{C_{p,\infty}^{w}}\lesssim
\sup_{r>0}\bigg(\frac{1}{r v(r)}\int_{B_r(0)} |x||f(x)|\, dx + \frac{1}{v(r)} \int_{\R^n \backslash B_r(0)} |f(x)|\, dx \bigg).
	\end{align}
Applying this estimate for $f_{j,k}$ and the weights $v_j, w_j$, $j=1,2$  (it is a routine to verify  that each couple $(v_j, w_j)$ satisfies both \eqref{EQweightedcampanato} and  \eqref{EQweightedcampanatocond2-}), we then establish
	\begin{align*}
		|\widehat{f}|_{C_{p,q}^w}&\lesssim \sum_{j=0}^1 \bigg(\sum_{k\in \Z} \bigg( \frac{w(2^k)}{w_j(2^k)} \sup_{m\in \Z} \Big( \frac{1}{2^m v_j(2^m)} \int_{B_{2^m}(0)} |x| |f_{j,k}(x)|\, dx\\
		&\phantom{=}+ \frac{1}{v_j(2^m)}
		\sum_{r=m}^\infty \int_{B_{2^{r+1}}(0)\backslash B_{2^r}(0)}|f_{j,k}(x)|\, dx\Big) \bigg)^q \bigg)^\frac{1}{q}=:A_0+A_1.
	\end{align*}
	
	By \eqref{EQalmostdecv_0} and \eqref{EQalmostincv_0}, we have that  $\frac{1}{v_0(r)}$ is almost increasing and $\frac{1}{rv_0(r)}$ is almost decreasing. Thus, we get
	\begin{align*}
		&\phantom{=}\sup_{m\in \Z} \bigg( \frac{1}{2^m v_0(2^m)} \int_{B_{2^m}(0)} |x| |f_{0,k}(x)|\, dx+ \frac{1}{v_0(2^m)}
		\sum_{r=m}^\infty \int_{B_{2^{r+1}}(0)\backslash B_{2^r}(0)}|f_{0,k}(x)|\, dx\bigg)\\
		&=\sup_{m \geq -k} \bigg( \frac{1}{2^m v_0(2^m)} \int_{B_{2^m}(0)} |x| |f_{0,k}(x)|\, dx+ \frac{1}{v_0(2^m)}
		\sum_{r=m}^\infty \int_{B_{2^{r+1}}(0)\backslash B_{2^r}(0)}|f_{0,k}(x)|\, dx\bigg)\\
		&\lesssim \sup_{m \geq -k} \bigg( \sum_{r=-k}^{m} \frac{1}{2^rv_0(2^r)} \int_{B_{2^{r+1}}(0)\backslash B_{2^r}(0)} |x| |f(x)|\, dx+
		\sum_{r=m}^\infty \frac{1}{v_0(2^r)} \int_{B_{2^{r+1}}(0)\backslash B_{2^r}(0)}|f(x)|\, dx\bigg)\\
		&\asymp   \sum_{r=-k}^\infty \frac{1}{v_0(2^r)}\int_{B_{2^{r+1}}(0)\backslash B_{2^r}(0)}|f(x)|\, dx.
	\end{align*}
	Similarly as above, $\frac{1}{v_1}$ is almost increasing and $\frac{1}{rv_1(r)}$ is almost decreasing. Hence, we obtain
	\begin{align*}
		&\phantom{=}\sup_{m\in \Z} \bigg( \frac{1}{2^m v_1(2^m)} \int_{B_{2^m}(0)} |x| |f_{1,k}(x)|\, dx+ \frac{1}{v_1(2^m)}
		\sum_{r=m}^\infty \int_{B_{2^{r+1}}(0)\backslash B_{2^r}(0)}|f_{1,k}(x)|\, dx\bigg)\\
		&=\sup_{m\leq -k} \bigg( \frac{1}{2^m v_1(2^m)} \int_{B_{2^m}(0)} |x| |f_{1,k}(x)|\, dx+ \frac{1}{v_1(2^m)}
		\sum_{r=m}^\infty \int_{B_{2^{r+1}}(0)\backslash B_{2^r}(0)}|f_{1,k}(x)|\, dx\bigg) \\
		&\lesssim \sup_{m\leq -k}   \bigg( \sum_{r=-\infty}^{m-1} \frac{1}{2^r v_1(2^r)} \int_{B_{2^{r+1}}(0)\backslash B_{2^r}(0)} |x||f(x)|\, dx+
		\sum_{r=m}^{-k} \frac{1}{v_1(2^r)} \int_{B_{2^{r+1}}(0)\backslash B_{2^r}(0)}|f(x)|\, dx\bigg)\\
		& \asymp  \sum_{r=-\infty}^{-k} \frac{1}{v_1(2^r)}\int_{B_{2^{r+1}}(0)\backslash B_{2^r}(0)}|f(x)|\, dx.
	\end{align*}
	Thus, by \eqref{EQweightedcampanato},
	\begin{align*}
		A_0&\lesssim \bigg(\sum_{k\in \Z} \Big(\frac{w(2^{-k})}{w_0(2^{-k})}  \sum_{r=k}^\infty \frac{1}{v_0(2^r)}\int_{B_{2^{r+1}}(0)\backslash B_{2^r}(0)}|f(x)|\, dx \Big)^q\bigg)^\frac{1}{q}\\
		&= \bigg(\sum_{k\in \Z} \Big(2^{-k\frac{n}{p}}w(2^{-k})v(2^k) 2^{k\gamma} \sum_{r=k}^\infty \frac{1}{v_0(2^r)}\int_{B_{2^{r+1}}(0)\backslash B_{2^r}(0)}|f(x)|\, dx \Big)^q\bigg)^\frac{1}{q}\\
		&\lesssim \bigg(\sum_{k\in \Z}\Big( \frac{1}{v(2^r)}\int_{B_{2^{r+1}}(0)\backslash B_{2^r}(0)}|f(x)|\, dx \Big)^q\bigg)^\frac{1}{q},
	\end{align*}
	where the last estimate follows from Hardy's inequality. Correspondingly, for $A_1$ we have
	\begin{align*}
		A_1&\lesssim\bigg(\sum_{k\in \Z}\Big(2^{-k\frac{n}{p}}w(2^{-k})v(2^k) 2^{-k\gamma}  \sum_{r=-\infty}^k \frac{1}{v_1(2^r)}\int_{B_{2^{r+1}}(0)\backslash B_{2^r}(0)}|f(x)|\, dx \Big)^q\bigg)^\frac{1}{q}\\
		&\lesssim \bigg(\sum_{k\in \Z}\Big( \frac{1}{v(2^r)}\int_{B_{2^{r+1}}(0)\backslash B_{2^r}(0)}|f(x)|\, dx \Big)^q\bigg)^\frac{1}{q},
	\end{align*}
	which concludes the proof.
\end{proof}

\begin{remark}
To define the weighted Campanato norm in the left hand side of \eqref{loctrMorrey}, we need to assume that $w\in \Xi_{p,q}^{n+p}$.
Let us see that this condition follows from \eqref{EQweightedcampanato} and \eqref{EQweightedcampanatocond2-}. Indeed,
to verify that $\big\| r^{\frac{n}{p}+1-\frac{1}{q}}w(r)\big\|_{L_q(0,1)} <\infty$, we get
\[
	\big\| r^{\frac{n}{p}+1-\frac{1}{q}}w(r)\big\|_{L_q(0,1)}^q\lesssim
\int_{1}^\infty \frac{dr}{r^{q +1}v^q(r)}
 \lesssim
 \int_{1}^\infty \frac{dr}{r^{1+\gamma}}
 \lesssim 1,
\]
where we have used that $r^{1-\gamma}v(r)$ is almost increasing. Similarly,
$
	\big\|r^{-\frac{1}{q}}w(r)\big\|_{{L_q(1,\infty)}}^q \lesssim \big\|r^{-\frac{n}{p}-\frac{1}{q}}v(1/r)^{-1}\big\|_{{L_q(1,\infty)}}^q
\lesssim1$, since $v$ is almost decreasing.

\end{remark}

\begin{remark}
	By carefully examining the proof of Theorem~\ref{THMcampanatoweighted} for $q=\infty$, we note that we
 have proved that
\begin{equation*}
	\frac{1}{r^\lambda} \bigg(\int_{B_{r}(0)}|\widehat{f}(x+\xi)-A_r \widehat{f}(\xi)|^{p}\,  d x\bigg)^\frac{1}{p}\lesssim
\sup _{s>0}\bigg( \frac{1}{s^{1-\alpha}} \int_{B_{s}(0)}|y||f(y)|\, d y+ s^{\alpha} \int_{\R^n \backslash B_s(0)}|f(y)|\, d y\bigg),
\end{equation*}
cf. \eqref{EQsupcampanato--+} with  $v(r)=r^{-\alpha}, w(r)=r^{-\lambda}$, $0<\alpha=\lambda-\frac{n}{p}<1$.
In fact this estimate also holds for the limiting parameters $\alpha=0$ and $\alpha=1$, since the restriction comes from the assumption \eqref{EQweightedcampanatocond2-}, which was not used (we have only used that $v$ is almost decreasing).
In particular, we have for $\alpha=0$
	\[
	| {\widehat{f}}|_{C_{p,\infty}^{n/p}}
 \lesssim \sup _{m\in \Z}\bigg( \frac{1}{2^{m}} \int_{B_{2^m}(0)}|y||f(y)|\, d y+ \int_{\R^n \backslash B_{2^m}(0)}|f(y)| \,d y\bigg) \asymp \|f\|_{L_1},
	\]
	which is trivial, since $\|\widehat{f}\|_{BMO}\lesssim \|\widehat{f}\|_{L_\infty}\leq \| f\|_{L_1}$, cf. \eqref{EQcampanatocharact}. On the other hand, for $\alpha=1$ (which corresponds to $\textrm{Lip } 1$, cf. \eqref{EQcampanatocharact} with $\lambda = \frac{n}{p}+1$), we have
	\[
	| {\widehat{f}}|_{C_{p,\infty}^{n/p+1}}\lesssim \sup _{m\in \Z}\bigg( \int_{B_{2^m}(0)}|y||f(y)|\, d y+ 2^{m} \int_{\R^n \backslash B_{2^m}(0)}|f(y)| \,d y\bigg) \asymp \||x|f\|_{L_1}.
	\]
\end{remark}

\vspace{0.4mm}
\appendix
\section{No Fourier inequalities between Morrey spaces}
\label{SECnofourierineq}
\subsection{Unweighted case}

\begin{theorem}
	For any $1\leq p,q<\infty$, $0\leq \nu \leq \frac{n}{q}$, and $0< \lambda\leq \frac{n}{p}$, there is no constant $C$ such that
	\begin{equation}
		\label{EQpittmorrey-nopitt}
		\| \widehat{f} \|_{M^{\nu}_q(\R^n)}\le C\|f \|_{M^{\lambda}_p(\R^n)}
	\end{equation}
	for every $f\in M^{\lambda}_p(\R^n)\cap L_1(\R^n)$.
\end{theorem}
\begin{proof}
	For $N\in\N$, define
	\[
	g_N(x)= \sum_{k=1}^N\chi_{(2^{k},2^{k}+1)}(x),\qquad x\in \R,
	\]
	and
	\[
	f_N(x_1,\ldots ,x_n)=\prod_{j=1}^n g_N(x_j), \qquad (x_1,\ldots , x_n)\in \R^n.
	\]
	On the one hand, for any $N\in \N$ one has
	$\label{EQfnnorm}
		\| f_N\|_{M_p^\lambda(\R^n)} \lesssim1. $
	On the other hand, it is clear that
	\[
	\widehat{g}_N(y)= \frac{1-e^{-2\pi iy}}{2\pi iy} \sum_{k=1}^N e^{-2\pi i 2^k y}, \qquad y\in \R,
	\]
	and by definition,
$
	\widehat{f}_N(x_1,\ldots ,x_n)=\prod_{j=1}^n \widehat{g}_N(x_j) ,\quad \text{and}\quad \|\widehat{f}_N\|_{L_q((0,1)^n)}=\| \widehat{g}_N\|^n_{L_q(0,1)}.
$
Thus, we deduce from Zygmund's result on lacunary trigonometric series \cite[Ch. V]{zygmund} that for any $1\leq q<\infty$,
	\[
	\| \widehat{g}_N\|_{L_q(0,1)} = \bigg( \int_0^1 \frac{|1-e^{-2\pi iy}|^q}{(2\pi y)^q} \bigg| \sum_{k=1}^N e^{-2\pi i 2^k y} \bigg|^q dy\bigg)^{\frac1q} \asymp \bigg\| \sum_{k=1}^{N} e^{-2\pi i 2^k y}\bigg\|_{L_{q}(0,1)}\asymp \sqrt{N},
	\]
	and hence
	\[
	\| \widehat{f}_N \|_{M^{\nu}_q(\R^n)} =\sup_{x\in \R^n, \,r>0} r^{-\nu}\| \widehat{f}_N\|_{L_q(B(x,r))}\gtrsim  \| \widehat{f}_N\|_{L_q((0,1)^n)}
\asymp N^{\frac{n}{2}}.
	\]
	Finally, assuming \eqref{EQpittmorrey-nopitt} is true, the result follows by contradiction.

\end{proof}

\subsection{Weighted case}
For simplicity we restrict ourselves to the one-dimensional case. We claim that there is no constant $C$ such that
\begin{equation}
	\label{EQweightedmorreymorrey}
	\||y|^{-\beta} \widehat{f} \|_{M^{\nu}_q(\R)} \leq C\||x|^\gamma f \|_{M^{\lambda}_p(\R)}
\end{equation}
holds
for every $f\in L_1(\R)$ for which the right hand side of \eqref{EQweightedmorreymorrey} is finite,
provided $ 0\leq \nu\leq \frac{1}{q}$ and $0<\lambda, \gamma\leq \frac{1}{p}$.

To construct a counterexample,  we make use of the so-called ``ultraflat polynomials'' \cite{kahane}, closely related to the Rudin-Shapiro construction \cite{rudin,shapiro}: for every $N\in \N$, there exists a trigonometric polynomial
\[
P_N(y)=\sum_{n=0}^{N} \varepsilon_n^N e^{-2\pi iny},\qquad \varepsilon_n^N\in \C, \quad |\varepsilon_n^N|=1,
\]
such that $|P_N(y)|\asymp \sqrt{N}$ uniformly in $(0,1)$. Given $N$, we define
\[
f_N(x)=\begin{cases}
	\varepsilon_n^N,&\text{if }x\in [n,n+1),\, n\in \{0,1,\ldots,  N\},\\
	0,&\text{otherwise}.
\end{cases}
\]
We easily obtain
\[
\| |x|^\gamma f_N\|_{M_p^\lambda} \asymp\sup_{0<r<N+1} r^{-\lambda} \bigg( \int^{N+1}_{N+1-r} x^{\gamma p}\, dx \bigg)^\frac{1}{p} \asymp N^{\gamma-\lambda+\frac{1}{p}}.
\]
On the other hand, since
\[
\widehat{f}_N(y) = \frac{1-e^{-2\pi iy}}{2\pi iy} P_N(y),\]
and $
\big| \frac{1-e^{-2\pi iy}}{2\pi iy}\big| \asymp 1
$
for $y\leq \frac{1}{2\pi}$, we get
\[
\| |x|^{-\beta}\widehat{f}_N \|_{M_q^\nu}\gtrsim \||x|^{-\beta}\widehat{f}_N \|_{L_q\big(\frac{1}{4\pi},\frac{1}{2\pi}\big)}\gtrsim\|P_N \|_{{L_q\big(\frac{1}{4\pi},\frac{1}{2\pi}}\big)}  \asymp N^\frac{1}{2}.
\]
If \eqref{EQweightedmorreymorrey} holds, then
\[
N^\frac{1}{2} \lesssim \| |x|^{-\beta}\widehat{f}_N \|_{M_q^\nu(\R)}\lesssim \||x|^\gamma f_N \|_{M^{\lambda}_p(\R)} \asymp N^{\gamma-\lambda+\frac{1}{p}}.
\]
Letting $N\to \infty$, we obtain that, necessarily,
\begin{equation}\label{cond vsp}
\frac{1}{2}\leq \gamma-\lambda+\frac{1}{p}.
\end{equation}
Since $\gamma\le1/p$, if $p\geq 4$, we have
\[
 \gamma-\lambda+\frac{1}{p} \le\frac{2}{p}-\lambda<\frac{2}{p}\le \frac{1}{2},
\]
which
contradicts \eqref{cond vsp}.
Thus,
\eqref{EQweightedmorreymorrey} does not hold for $p\geq 4$.

{\bf Acknowledgements}. The authors would like to thank \'{O}scar Dom\'inguez and Miquel Saucedo
for their helpful advices.

\end{document}